\documentclass[11pt,leqno]{amsart}

\usepackage{latexsym}
\usepackage{amssymb}
\usepackage{amsmath}
\usepackage{amsthm}
\usepackage{verbatim}
\usepackage{pdfsync}
\usepackage[margin=1.2in]{geometry}
\usepackage{tikz}
\usepackage[curve,matrix,arrow,frame,tips]{xy}

\numberwithin{equation}{section}  %%% this one seems to control what happens!!

\newcommand{\TT}{T}

\usepackage{color}
\usepackage{graphicx}

\makeatletter
\newcommand*{\bigcdot}{}% Check if undefined
\DeclareRobustCommand*{\bigcdot}{%
  \mathbin{\mathpalette\bigcdot@{}}%
}
\newcommand*{\bigcdot@scalefactor}{.5}
\newcommand*{\bigcdot@widthfactor}{1.15}
\newcommand*{\bigcdot@}[2]{%
  \sbox0{$#1\vcenter{}$}% math axis
  \sbox2{$#1\cdot\m@th$}%
  \hbox to \bigcdot@widthfactor\wd2{%
    \hfil
    \raise\ht0\hbox{%
      \scalebox{\bigcdot@scalefactor}{%
        \lower\ht0\hbox{$#1\bullet\m@th$}%
      }%
    }%
    \hfil
  }%
}
\makeatother

\begin{document}

%%% These commands yield the numbering conventions used in the Kyoto notes.
%\renewcommand{\thesection}{\Roman{section}}
%\renewcommand{\thesubsection}{\Roman{section}.\arabic{subsection}}

%new macros by Tonghai

\newcommand{\zxz}[4]{\begin{pmatrix} #1 & #2 \\ #3 & #4 \end{pmatrix}}
\newcommand{\abcd}{\zxz{a}{b}{c}{d}}
\newcommand{\kzxz}[4]{\left(\begin{smallmatrix} #1 & #2 \\ #3 &
#4\end{smallmatrix}\right) }
\newcommand{\kabcd}{\kzxz{a}{b}{c}{d}}

%%%%%%%%%%%%%%% This is the macros file for alll chapters of KYR
%%%%% last edited:  ssk  9/26/04

%%%%%%%%%%%%%%%  macros

%%%% mathbb

\newcommand{\A}{{\mathbb A}}
\newcommand{\C}{{\mathbb C}}
\newcommand{\F}{{\mathbb F}}
\newcommand{\G}{{\mathbb G}}
\newcommand{\R}{{\mathbb R}}
\newcommand{\Q}{{\mathbb Q}}
\newcommand{\X}{{\mathbb X}}
\newcommand{\Z}{{\mathbb Z}}
\newcommand{\HZ}{\widehat{\Z}}

%%%% mathrm

\newcommand{\rom}[1]{\text{\rm #1}}
\renewcommand{\roman}{\rm}

\newcommand{\Aut}{\text{\rm Aut}}
\newcommand{\CH}{\widehat{\text{\rm CH}}}
\newcommand{\cha}{{\text{\rm char}}}
\newcommand{\CHe}{\text{\rm CHeeg}}
\newcommand{\degh}{\widehat{\text{\rm deg}}}
\newcommand{\degH}{\widehat{\text{\rm deg}}}    %%% redundant def
\newcommand{\diag}{{\text{\rm diag}}}
\newcommand{\Diff}{\text{\rm Diff}}
\newcommand{\disc}{\text{\rm discr}}
\renewcommand{\div}{\text{\rm div}}
\newcommand{\divh}{\widehat{\text{\rm div}}}
\newcommand{\DS}{\text{\rm DS}}
\newcommand{\Ei}{\text{\rm Ei}}
\newcommand{\End}{\text{\rm End}}
\newcommand{\ev}{{\text{\rm ev}}}
\newcommand{\Gal}{\text{\rm Gal}}
\newcommand{\GL}{\text{\rm GL}}
\newcommand{\GSpin}{\text{\rm GSpin}}
\newcommand{\Hom}{\text{\rm Hom}}
\newcommand{\hor}{{\text{\rm horiz}}}
\newcommand{\id}{\text{\rm id}}
\newcommand{\im}{\text{\rm im}}
\renewcommand{\Im}{\text{\rm Im}}
\newcommand{\inv}{{\text{\rm inv}}}
\newcommand{\Jac}{\text{\rm Jac}}
\newcommand{\Leray}{{\mathrm L}}
\newcommand{\Lie}{\text{\rm Lie}}
\newcommand{\Mp}{\text{\rm Mp}}
\newcommand{\mult}{\text{\rm mult}}
\newcommand{\MW}{\text{\rm MW}}
\newcommand{\MWt}{\widetilde{\MW}}
\newcommand{\new}{\text{\rm new}}
\newcommand{\Nm}{\text{\rm Nm}}
\newcommand{\ord}{\text{\rm ord}}
\newcommand{\PGL}{\text{\rm PGL}}
\newcommand{\Pic}{\text{\rm Pic}}
\newcommand{\Pich}{\widehat{\text{\rm Pic}}}
\newcommand{\pr}{\text{\rm pr}}
\newcommand{\ra}{\text{\rm ra}}
\newcommand{\Rao}{\mathrm R}
\renewcommand{\Re}{\text{\rm Re}}
\newcommand{\sgn}{\text{\rm sgn}}
\newcommand{\sig}{\text{\rm sig}}
\newcommand{\SL}{\text{\rm SL}}
\newcommand{\SO}{\text{\rm SO}}
\newcommand{\Sp}{\text{\rm Sp}}
\newcommand{\Spec}{\text{\rm Spec}\, }
\newcommand{\Spf}{\text{\rm Spf}}
\newcommand{\supp}{\text{\rm supp}}
\newcommand{\Sym}{{\text{\rm Sym}}}
\newcommand{\tr}{\text{\rm tr}}
\newcommand{\type}{\text{\rm type}}
\newcommand{\Ver}{\text{\rm Vert}}
\newcommand{\vol}{\text{\rm vol}}
\newcommand{\Wald}{\text{\rm Wald}}

%%%% cals

\newcommand{\Cal}{\mathcal}     %%% this makes the old \Cal valid

\newcommand{\AHH}{\hat{\Cal A}}   % used??
\newcommand{\CHH}{\hat{\Cal C}}
\newcommand{\MM}{\Cal D}          % redefined!!!
\newcommand{\MMb}{\MM^\bullet}
\newcommand{\ssplit}{\text{\bf split}}
\newcommand{\whcc}{\widehat{\Cal C}}
\newcommand{\CO}{\mathcal O}
\newcommand{\COH}{\widehat{\CO}}
\newcommand{\M}{\Cal M}
\newcommand{\OB}{\Cal O_B}
\newcommand{\XX}{\mathcal X}
\newcommand{\bXX}{\bar\XX}
\newcommand{\wc}{\hat{\Cal C}}
\newcommand{\wch}{\wc^{\text{\rm hor}}}
\newcommand{\ZZ}{\Cal Z}
\newcommand{\ZH}{\widehat{\Cal Z}}   %%% redundant def's
\newcommand{\Zh}{\widehat{\Cal Z}}
\newcommand{\ZZh}{\ZZ^{\text{\rm hor}}}
\newcommand{\ZZv}{\ZZ^{\text{\rm ver}}}
\newcommand{\ZZhh}{\Zh^{\text{\rm hor}}}
\newcommand{\ZZhv}{\Zh^{\text{\rm ver}}}

%%%% math spacing

\newcommand{\nass}{\noalign{\smallskip}}
\newcommand{\snass}{\noalign{\vskip 2pt}}
\newcommand{\tent}[1]{ \vphantom{\vbox to #1pt{}} }   %%% !!!!

%%%% math fonts

\newcommand{\scr}{\scriptstyle}
\newcommand{\disp}{\displaystyle}

\font\cute=cmitt10 at 12pt
\font\smallcute=cmitt10 at 9pt
\newcommand{\kay}{{\text{\cute k}}}
\newcommand{\smallkay}{{\text{\smallcute k}}}

\renewcommand{\a}{\alpha}
\renewcommand{\b}{\beta}
\newcommand{\e}{\epsilon}
\renewcommand{\l}{\lambda}
\renewcommand{\L}{\Lambda}
\renewcommand{\o}{\omega}
\renewcommand{\O}{\Omega}
\renewcommand{\P}{\Phi}
\newcommand{\ph}{\varphi}
\newcommand{\phih}{\widehat{\phi}}
\newcommand{\wphi}{\widehat{\phi}}
\newcommand{\phit}{\widetilde{\phi}}
\newcommand{\s}{\sigma}
\newcommand{\vth}{\vartheta}

%%%% from Chapter VIII

%%\newcommand{\Gt}{\widetilde{G}}    %%%%  tilde's removed 6/20/04
%\newcommand{\Gt}{G}
%\newcommand{\Ph}{\Phi}
%\newcommand{\pht}{\widetilde{\phi}}
%%\newcommand{\Pht}{\widetilde{\Phi}}%%%%  tilde's removed 6/20/04
%\newcommand{\Pht}{\Phi}
%%\newcommand{\Pt}{\widetilde{P}}
%\newcommand{\Pt}{P}                 %%%%  tilde's removed 6/20/04
%%\newcommand{\Kt}{\widetilde{K}}    %%%%  tilde's removed 6/20/04
%\newcommand{\Kt}{K}
%%\newcommand{\It}{\widetilde{I}}    %%%%  tilde's removed 6/20/04
%\newcommand{\It}{I}
%\newcommand{\Jt}{\widetilde{J}}
%\newcommand{\lt}{\widetilde{\l}}
%\newcommand{\vp}{\varpi}
%

\newcommand{\Pt}{P}
\newcommand{\Ph}{\P}
\newcommand{\Pht}{\tilde \P}   %%%%%%%%    ****** conflicts *******
\newcommand{\Kt}{K}           %%%%  tilde's removed 6/20/04
\newcommand{\Mt}{M}
%%%%%%

%\newcommand{\Ph}{\Phi}      %%%%%%%%    ****** conflicts *******   temp %'ed
\newcommand{\pht}{\widetilde{\phi}}
\newcommand{\It}{I}
\newcommand{\Jt}{\widetilde{J}}
\newcommand{\lt}{\widetilde{\l}}
\newcommand{\vp}{\varpi}

\newcommand{\bom}{{\boldsymbol{\o}}}
\newcommand{\hbom}{\widehat{\bom}}
\newcommand{\ff}{{\bold f}}
\newcommand{\fsp}{\boldsymbol{f}_{\!\rm sp}}
\newcommand{\fev}{\boldsymbol{f}_{\!\rm ev}}
\newcommand{\fb}{\boldsymbol{f}}
\newcommand{\J}{\und{J}'}
\newcommand{\JJ}{\bold J'}
\newcommand{\V}{\bold V}
\newcommand{\xx}{\bold x}

\newcommand{\g}{{\mathfrak g}}
\renewcommand{\H}{\mathfrak H}

%%%%  math macros

\newcommand{\back}{\backslash}
\newcommand{\CT}[1]{\operatornamewithlimits{CT}_{#1}}
\renewcommand{\d}{\partial}
\newcommand{\db}{\bar\partial}
\newcommand{\dbar}{\bar{\partial}}
\newcommand{\gs}[2]{\langle \,#1,#2\,\rangle}
\newcommand{\Gt}{G}
\newcommand{\hfal}{h_{\text{\rm Fal}}}
\newcommand{\II}{\int^\bullet}
\newcommand{\isoarrow}{\ {\overset{\sim}{\longrightarrow}}\ }
\newcommand{\lisoarrow}{\ {\overset{\sim}{\longleftarrow}}\ }
\newcommand{\limdir}[1]{\underset{\underset{#1}{\rightarrow}}{\lim}}
\newcommand{\lan}{\operatorname{\langle}\hskip .5pt}
\newcommand{\ran}{\,\operatorname{\rangle}}
\newcommand{\lra}{\longrightarrow}
\newcommand{\doublelra}{\ {\overset{\scr\lra}{\scr\lra}}\ }
\newcommand{\nat}{\natural}
\newcommand{\notmid}{\mkern-5mu\not\mkern5mu\mid}
\newcommand{\Optoc}{\text{\rm Opt}(O_{c^2d},O_B)}
\newcommand{\psim}{\psi^{-}}
\newcommand{\qeq}{\ \overset{??}{=}\ }
\newcommand{\sh}{\sharp}
\newcommand{\thCH}{\theta^{\text{\rm ar}}}
\newcommand{\wht}{\widehat{\theta}}     %%% a replacement
\newcommand{\triv}{1\!\!1}
\renewcommand{\tt}{\otimes}
\newcommand{\und}[1]{\underline{#1}}
\newcommand{\z}{z}  %%% symbol used for the central sign 

\newcommand{\thMW}{\theta^{\text{\rm ar}}}
\newcommand{\tph}{\widetilde{\widehat\phi_1}}
\newcommand{\Pet}{\text{\rm Pet}}

%%%%% from Chapter IX

%%\newcommand{\Gt}{\widetilde{G}}    %%%%  tilde's removed 6/20/04
%\newcommand{\Gt}{G}
%\newcommand{\Ph}{\Phi}
%\newcommand{\pht}{\widetilde{\phi}}
%%\newcommand{\Pht}{\widetilde{\Phi}}%%%%  tilde's removed 6/20/04
%\newcommand{\Pht}{\Phi}
%%\newcommand{\Pt}{\widetilde{P}}
%\newcommand{\Pt}{P}                 %%%%  tilde's removed 6/20/04
%%\newcommand{\Kt}{\widetilde{K}}    %%%%  tilde's removed 6/20/04
%\newcommand{\Kt}{K}
%%\newcommand{\It}{\widetilde{I}}    %%%%  tilde's removed 6/20/04
%\newcommand{\It}{I}
%\newcommand{\Jt}{\widetilde{J}}
%\newcommand{\lt}{\widetilde{\l}}
%\newcommand{\vp}{\varpi}

%%%%%%%%%%%

%%%%  hacking

\newcommand{\thing}{ \raisebox{-6.4pt}{$\tilde{\tilde{}}$}  }   %%% some hacking from 4/14/04
\newcommand{\smallthing}{ \raisebox{-4.4pt}{$\scr\tilde{\tilde{}}$}  }
\newcommand{\ttilde}[1]{\overset{\smash{\thing}}{#1}}
\newcommand{\smallttilde}[1]{\overset{\smash{\smallthing}}{#1}}
\newcommand{\downhookarrow}{\hbox{$\downarrow\hskip -6.1pt\raisebox{6pt}{$\cap$}$}}

%%%% general formating

%\newcommand{\bysame}{\makebox[1.2cm][s]{\hrulefill ,\ }}   %%%%  \bysame is defined already
%\newcommand{\bysame}{$\underline{\text{\hbox to.5in{}}}$}   %%% was used in Textures
\providecommand{\bysame}{\makebox[3em]{\hrulefill}\thinspace}   %%% a fix:  cf. p 322 of Graetzer
\newcommand{\hfb}{\hfill\break}
\newcommand{\margincom}[1]{\marginpar{\bf\raggedright #1}}
\newcommand{\Sec}{\S}

%%%%%%%%%%%%%%

\numberwithin{equation}{section}
\setcounter{section}{0}
\setcounter{MaxMatrixCols}{15}

%%%%%%%%%%%%%%

\newtheorem{theo}{Theorem}[section]
\newtheorem{lem}[theo]{Lemma}
\newtheorem{prop}[theo]{Proposition}
\newtheorem{cor}[theo]{Corollary}
\newtheorem*{main}{Main Theorem}
\newtheorem*{atheo}{Theorem A}
\newtheorem*{btheo}{Theorem B}
\newtheorem{conj}[theo]{Conjecture}
\newtheorem{rem}[theo]{Remark}      %%% seems not to exist in compositio.cls ???
\newtheorem{defn}[theo]{Definition}

\newcommand{\OO}{\text{\rm O}}
\newcommand{\UU}{\text{\rm U}}

\newcommand{\OK}{O_{\smallkay}}
\newcommand{\DI}{\mathcal D^{-1}}

\newcommand{\pre}{\text{\rm pre}}

\newcommand{\Bor}{\text{\rm Bor}}
\newcommand{\Rel}{\text{\rm Rel}}
\newcommand{\rel}{\text{\rm rel}}
\newcommand{\Res}{\text{\rm Res}}
\newcommand{\TG}{\widetilde{G}}

\newcommand{\OL}{O_{\Lambda}}
\newcommand{\OLB}{O_{\Lambda,B}}

\newcommand{\p}{\varpi}

\newcommand{\cutter}{\vskip .1in\hrule\vskip .1in}

\parindent=0pt
\parskip=10pt
\baselineskip=15pt

\newcommand{\PP}{\mathcal P}
\renewcommand{\OO}{\mathcal O}
\newcommand{\BB}{\mathbb B}
\newcommand{\OBB}{O_{\BB}}
\newcommand{\Max}{\text{\rm Max}}
\newcommand{\Opt}{\text{\rm Opt}}
\newcommand{\OH}{O_H}

\newcommand{\phhat}{\widehat{\phi}}
\newcommand{\thetahat}{\widehat{\theta}}

\newcommand{\lbold}{\text{\boldmath$\l$\unboldmath}}
\newcommand{\abold}{\text{\boldmath$a$\unboldmath}}
\newcommand{\cbold}{\text{\boldmath$c$\unboldmath}}
\newcommand{\ebold}{\text{\boldmath$e$\unboldmath}}
\newcommand{\aabold}{\text{\boldmath$\a$\unboldmath}}
\newcommand{\gbold}{\text{\boldmath$g$\unboldmath}}
\newcommand{\hbold}{\text{\boldmath$h$\unboldmath}}
\newcommand{\obold}{\text{\boldmath$\o$\unboldmath}}
\newcommand{\fbold}{\text{\boldmath$f$\unboldmath}\!}
\newcommand{\rbold}{\text{\boldmath$r$\unboldmath}}
\newcommand{\ffbold}{\und{\fbold}}
\newcommand{\sbold}{\text{\boldmath$\s$\unboldmath}}
\newcommand{\tbold}{\text{\boldmath$t$\unboldmath}}
\newcommand{\qbold}{\text{\boldmath$q$\unboldmath}}
\newcommand{\gammabold}{\text{\boldmath$\gamma$\unboldmath}}

\newcommand{\boldCC}{\text{\boldmath$\mathcal C$\unboldmath}}
\newcommand{\kbold}{\text{\boldmath$k$\unboldmath}}
\newcommand{\fbolds}{\fbold\,}

\newcommand{\deltaBB}{\delta_{\BB}}%{\delta\!\!\!\delta}
\newcommand{\kappaBB}{\kappa_{\BB}}%{\kappa\!\!\!\kappa}
\newcommand{\aboldBB}{\abold_{\BB}}
\newcommand{\lboldBB}{\lbold_{\BB}}
\newcommand{\gboldBB}{\gbold_{\BB}}
\newcommand{\bbold}{\text{\boldmath$\b$\unboldmath}}

\newcommand{\phbold}{\text{\boldmath$\ph$\unboldmath}}

\newcommand{\fff}{\phi}
\newcommand{\spp}{\text{\rm sp}}

\newcommand{\pob}{\mathfrak p_{\bold o}}
\newcommand{\kob}{\mathfrak k_{\bold o}}
\newcommand{\gob}{\mathfrak g_{\bold o}}
\newcommand{\pobp}{\mathfrak p_{\bold o +}}
\newcommand{\pobm}{\mathfrak p_{\bold o -}}

\newcommand{\mmm}{\mathfrak m}

%%%%%%%%%%%   saved from old Green function document

\newcommand{\bb}{\frak b}

\newcommand{\bbbold}{\text{\boldmath$b$\unboldmath}}

\renewcommand{\ll}{\,\frak l}
\newcommand{\uC}{\underline{\Cal C}}
\newcommand{\uZZ}{\underline{\ZZ}}
\newcommand{\B}{\mathbb B}
\newcommand{\CL}{\text{\rm Cl}}

\newcommand{\pp}{\frak p}

\newcommand{\OKp}{O_{\smallkay,p}}

\renewcommand{\top}{\text{\rm top}}

\newcommand{\bF}{\bar{\mathbb F}_p}

%%%%%%
\newcommand{\beq}{\begin{equation}}
\newcommand{\eeq}{\end{equation}}

%%%%%%

\newcommand{\Dl}{\Delta(\l)}
\newcommand{\mm}{{\bold m}}

\newcommand{\FD}{\text{\rm FD}}
\newcommand{\LDS}{\text{\rm LDS}}

\newcommand{\dcM}{\dot{\Cal M}}
\newcommand{\bpm}{\begin{pmatrix}}
\newcommand{\epm}{\end{pmatrix}}

\newcommand{\GW}{\text{\rm GW}}

\newcommand{\uk}{\bold k}
\newcommand{\uo}{\text{\boldmath$\o$\unboldmath}}

\newcommand{\ww}{\bold w}
\newcommand{\yy}{\bold y}

\newcommand{\Om}{\text{\boldmath$\Omega$\unboldmath}}

\newcommand{\phnat}{{}^\nat\phi}
\newcommand{\om} {\text{\boldmath$\o$\unboldmath}}
\newcommand{\omnat}{\breve\om}

\newcommand{\bino}[2]{{#1\choose{#2}}}

\renewcommand{\ev}{\ell^\vee}
\newcommand{\kk}{\frak k}

\newcommand{\gw}[2]{\langle\langle #1,#2 \rangle\rangle}
\newcommand{\gwrm}{\text{\rm gw}}

\newcommand{\cpar}{\sbold}

%\newcommand{\beq}{\begin{equation}}
%\newcommand{\eeq}{\end{equation}}

%% from bonn notes
\newcommand{\N}{\Cal N}
\newcommand{\GG}{\Cal G}
\newcommand{\GGu}{\und{\GG}}
\newcommand{\LGG}{\Cal L \Cal G}

\newcommand{\car}{\text{\rm char}}

\newcommand{\Con}{\text{\rm Con}}

\newcommand{\h}{\mathfrak h}
\newcommand{\prz}{{\bf prz}}
\newcommand{\hc}{\text{\rm hc}}

\renewcommand{\z}{\mathfrak z}
\newcommand{\x}{\und{x}}
\newcommand{\har}{\text{\rm har}}

\newcommand{\qq}{\qbold}
\newcommand{\qqq}{\mathfrak q}
\newcommand{\Skew}{\text{\rm Skew}}

\newcommand{\Dp}{D^+}
\newcommand{\Djp}{D^{(j),+}}
\newcommand{\ch}{\text{\rm CH}}
\renewcommand{\ra}{\rightarrow}
\newcommand{\wtg}{\widetilde{G'}}

\newcommand{\Sb}{\mathcal S^{\bigcdot}}

%%%% added for the embedding trick %%%%%
\newcommand{\sfs}{\Sb_F}
\newcommand{\sfsnull}{\Sb_F\cup \{0\}}
\newcommand{\sfp}{\mathcal S_{F,\ge0}^{\vee}}
\newcommand{\FFS}{\operatorname{FFS}}
\newcommand{\FFSS}{\FFS^{\bigcdot}}%{\FFS^\star}

%%%%%%%%%%%%

\newcommand{\now}{\count0=\time 
\divide\count0 by 60
\count1=\count0
\multiply\count1 by 60
\count2= \time
\advance\count2 by -\count1
\the\count0:\the\count2}

%\centerline{\it\hfill\today:\ \now}

\title{Remarks on generating series for special cycles}

\author{ Stephen S. Kudla}
\address{Department of Mathematics\\University of Toronto\\
40 St George St, BA6290\\
Toronto, ON M5S 2E4\\
Canada
}
\email{skudla@math.toronto.edu}
\thanks{Research supported by an NSERC Discovery Grant}
\keywords{Orthogonal Shimura varieties, algebraic cycles, Hilbert-Siegel modular forms}
\subjclass[2010]{Primary: 14C25, 11F27; Secondary: 11F46, 14G35}
\begin{abstract}  In this note, we consider special algebraic cycles on the Shimura variety $S$ 
associated to a quadratic space $V$ over a totally real field $F$, $|F:\Q|=d$,  of signature
$$((m,2)^{d_+},(m+2,0)^{d-d_+}), \qquad 1\le d_+<d.$$
For each $n$, $1\le n\le m$, 
there are special cycles $Z(T)$ in $S$, of codimension $nd_+$,  indexed by totally positive semi-definite matrices with coefficients in the ring of integers $O_F$.
The generating series for the 
classes of these cycles in the cohomology group $H^{2nd_+}(S)$ are Hilbert-Siegel modular forms of parallel weight $\frac{m}2+1$. 
One can form analogous generating series for the classes of the special cycles in the Chow group $\text{\rm CH}^{nd_+}(S)$. 
For $d_+=1$ and $n=1$, the modularity of these series was proved by Yuan-Zhang-Zhang.  
In this note we prove the following: 
Assume the Bloch-Beilinson conjecture on the injectivity 
of  Abel-Jacobi maps. Then the Chow group valued generating series for special cycles of codimension $nd_+$ on $S$ is modular for all $n$ with $1\le n\le m$. 
\end{abstract} 

\maketitle

\section{Introduction}

The goal of the present note is to probe the limits of what we know about certain special cycle generating series. 
Suppose that $V$ is a quadratic space over a totally real field $F$ of degree $d$ such that the signature of $V$ 
is 
\beq\label{sig-V}
((m,2), \dots, (m,2), (m+2,0),\dots ,(m+2,0)) = ((m,2)^{d_+},(m+2,0)^{d-d_+}), \quad d_+>0.
\eeq
We also suppose that $m>0$. The special cycles in the associated orthogonal Shimura variety $S$
have codimensions $nd_+$ for $1\le n\le m$. Thus there is a significant difference between the case $d_+=1$, where there are special cycles of every codimension, 
and the case $d_+>1$, where there are not.  The modularity of Chow 
group\footnote{We work with Chow groups with rational coefficients and write $\ch(X)$ rather than $\ch(X)_\Q$.} valued generating series in the case $d_+=1$ is established in many 
cases; we will review what is known in a moment. However, when $d_+>1$ the modularity of such series 
is more problematic, due to a lack of any 
systematic source of relations.  

As a concrete example\footnote{This example arose in discussions with Luis Garcia and Jan Bruinier
and was 
the initial motivation for this paper.}, suppose that $B$ is a totally indefinite division quaternion algebra over a real quadratic field $F$ which is  not a base change from $\Q$. The space $V$ of trace zero elements in $B$ with quadratic form given by the reduced norm has signature $((1,2),(1,2))$. 
The special cycles on the associated Shimura surface $S$ are $0$-cycles indexed by totally positive elements of the ring of integers 
$O_F$, and the generating series for their degrees is a Hilbert modular form of weight $(3/2,3/2)$.  The modularity of the generating series for 
their classes in $\ch^2(S)$ is not known however and would depend on the existence of many relations among these  $0$-cycles.  Recall 
that relations arise from collections $\{(C_i,f_i)\}$ where $C_i$ is a curve on $S$,  $f_i$ is a meromorphic function on $C_i$, and the 
$0$-cycle on $S$ given by $\sum_i \div_{C_i}(f_i)$  is zero in $\ch^2(S)$. 
But there are no evident curves on $S$ and hence any relations among the special $0$-cycles have no evident modular construction 
and would have to arise in some other way. Of course, the situation is the same whenever $d_+>1$,  since there are no special cycles of 
codimension $nd_+\!\!-1$ available to generate relations.

The modularity of generating series for certain divisor classes on the orthogonal Shimura variety $S$ 
associated to a quadratic space $V$ of signature $(m,2)$ over $\Q$ was proved by Borcherds, 
\cite{Bo1}.  His proof depends on the existence of a sufficient supply of meromorphic functions on $S$ with explicitly known divisors, constructed 
by means of his regularized theta lift.  They provide the relations among the special divisors in $\text{\rm CH}^1(S)$ and these relations among the 
coefficients of the generating series imply modularity.  The problem of showing modularity of analogous generating series for special cycles of 
higher codimension,  series with coefficients in $\text{\rm CH}^n(S)$, was suggested in \cite{K.duke}.  
In his thesis \cite{wei.zhang.thesis}, Wei Zhang showed that such series are indeed the $\qq$-expansions of Siegel modular forms of genus $n$ under the assumption that the 
series are convergent. His proof is based on an induction, beginning with the result of Borcherds for divisors,  and a calculation of the 
Fourier-Jacobi coefficients of the generating series.  Subsequently, Bruinier and Westerholt-Raum \cite{bruinier-raum} established the required convergence
by an argument based on an analysis of the dimensions of the spaces of Jacobi forms that arise as Fourier-Jacobi coefficients. 
Such an argument has its roots in the work of \cite{aoki} and \cite{ibukiyama}. 

Over a totally real field $F$ and in the case $d_+=1$,  the generating series for special cycles of codimension $n$ was considered in \cite{K.duke}, 
where the modularity of its image under the cycle class map to the (Betti) cohomology group $H^{2n}(S)$ is shown to be a consequence of the results of \cite{KM1}, \cite{KM2}, \cite{KM3}.
Using the vanishing of the first Betti number of such varieties\footnote{The low dimensional exceptions are handled by the embedding trick which we explain in Section~\ref{section8}.}, 
it is shown in \cite{YZZ} that modularity of the $\ch^1(S)$-valued generating series for special divisors follows from the modularity of the $H^2(S)$-valued series. 
Moreover, it is shown in \cite{YZZ}  that the inductive argument of \cite{wei.zhang.thesis} can be carried over to the $d_+=1$ case and yields modularity of the $\ch^n(S)$-valued 
generating series, again assuming the convergence of the series.  At present, no analogue of the Bruinier-Westerholdt-Raum result is available for totally real fields of degree $d>1$, and so
modularity of the $\ch^n(S)$-valued generating series in remains open. 

In the present paper we consider the case in which $d_+$ is arbitrary.  Since we want to avoid a discussion of compactifications, we will assume that $V$ is anisotropic and hence, 
when $m\ge 3$, that $d_+<d$.  The definition of both the connected and weighted special cycles given in \cite{K.duke} for $d_+=1$ goes over to the general case with almost no change. 
One important difference, however, is that the role played by the hyperplane section line bundle in Section 6 of \cite{K.duke} is now played by a class
$\cbold_S\in\ch^{d_+}(S)$ constructed as a product of the Chern classes of inverses of tautological bundles. The weighted special cycles $[Z(\TT,\ph)]\in \ch^{nd_+}(S)$ are indexed by pairs $(\TT,\ph)$ 
where $\TT\in \Sym_n(F)$ is
positive semi-definite\footnote{We write $\Sym_n(F)_{\ge0}$ for the space of such totally positive semi-definite matrices.} at each archimedean place of $F$
and $\ph\in S(V(\A_f)^n)$ is a Schwartz function on the finite ad\`eles of $V$.  We establish the analogues for general $d_+$ of various properties of these cycles 
proved, for $d_+=1$,  in \cite{K.duke} and 
in \cite{YZZ}. For example, there is a product formula in the Chow ring $\ch^\bullet(S)$, Proposition~\ref{prop5.1}, 
\beq\label{intro-prod-form}
Z(\TT_1,\ph_1)\cdot Z(\TT_2,\ph_2) = \sum_{\substack{ \TT \in \Sym_{n_1+n_2}(F)_{\ge0}\\ \snass \TT = \bpm \scr \TT_1&*\\ \scr{}^t*&\scr\TT_2\epm}} Z(\TT,\ph_1\tt\ph_2)\ \in \ch^{(n_1+n_2)d_+}(S).
\eeq
In the case $d_+=1$, this is proved in \cite{YZZ}, while the analogous cup product formula for images in cohomology is proved in \cite{K.duke}. 
The proof we give in Section~\ref{section5} for general $d_+$ makes use of the intersection theory from Fulton \cite{fulton.book},  a computation of excess bundles, and some Jaffee Lemma arguments, cf. Lemma~\ref{jaffee.lem} and Proposition~\ref{prop-good-cover},  which allow us to pass to suitable covers to 
achieve regular embeddings. Also, there is a formula for the pullback of special cycles to Shimura subvarieties associated to totally positive definite subspaces $U$ of $V$. 
This formula plays a key role in the embedding trick.

The generating series for special cycles of codimension $nd_+$ is the formal $\qq$-series 
\beq\label{ch-gen} 
\phi_n(\tau,\ph, S) = \sum_{\TT\in \Sym_n(F)_{\ge0}} [Z(\TT,\ph)]\, \qq^\TT\ \in \ \ch^{nd_+}(S)[[\qq]],
\eeq
where $\tau = (\tau_1, \dots, \tau_d)\in \H_n^d$, $\ph\in S(V(\A_f)^n)$, and 
\beq\label{def-qT}
\qq^T = e( \sum_{j=1}^d \tr(\s_j(T)\tau_j)).
\eeq
Here $\H_n$ is the Siegel space of genus $n$ and $\Sigma= \{\s_j\}_{1\le j\le d}$ is the set of archimedean embeddings of $F$. 

The product formula implies the following identity for the formal generating series:
\beq\label{prod-gen-series}
\phi_n(\bpm \tau_1&{}\\{}&\tau_2\epm, \ph_1\tt\ph_2,S) = \phi_{n_1}(\tau_1,\ph_1,S)\cdot \phi_{n_2}(\tau_2,\ph_2,S),
\eeq
whose analogue for generating series for cohomology classes was proved in \cite{K.duke}, for $d_+=1$,  using the theta series. 
Here the product on the right side is take in the Chow ring of $S$. 

The  series (\ref{ch-gen}) is said to be modular if, for every complex valued linear functional on $\ch^{nd_+}(S)$ the 
formal Fourier series\footnote{This is made more precise in Section~\ref{section8}.}  
\beq\label{ch-gen-lambda} 
\phi_n(\tau,\ph, S,\lambda) = \sum_{\TT\in \Sym_n(F)_{\ge0}} \lambda\big(\ [Z(\TT,\ph)]\ \big)\, \qq^\TT\ \in \ \C[[\qq]],
\eeq
is absolutely convergent and the resulting holomorphic function on $\H_n^d$ is a Hilbert-Siegel modular form.

For example, the image 
$$\phi_n(\tau,\ph, S,\text{\rm cl}) = \sum_{\TT\in \Sym_n(F)_{\ge0}} \text{\rm cl}[Z(\TT,\ph)]\, \qq^\TT\ \in \, H^{2nd_+}(S)[[\qq]]$$
of this series under the cycle class map $\text{\rm cl}=\text{\rm cl}_{nd_+}:\ch^{nd_+}(S) \ra H^{2nd_+}(S)$ 
is the $\qq$-expansion of a Hilbert-Siegel modular form of parallel weight $(\frac{m}2+1, \dots, \frac{m}2+1)$, again by
the results of \cite{KM1}, \cite{KM2}, \cite{KM3}. 
 Of course, if the cycle class map happens to be injective, then the modularity of (\ref{ch-gen}) follows from this immediately. 
 As observed in \cite{YZZ}, such injectivity would 
 result from a combination of the Bloch-Beilinson conjecture, which predicts that the kernel of $\text{\rm cl}_{nd_+}$ maps injectively\footnote{Recall that our Chow groups are taken with rational 
 coefficients and all of our varieties and special cycles are defined over number fields.} to the intermediate 
 Jacobian $J^{nd_+}(S)$ under the Abel-Jacobi map, and the vanishing of $H^{2nd_+-1}(S)$, which implies that  $J^{nd_+}(S)=0$. 
 
 A main result of this paper is that we can use a variant of this observation to obtain the following. 
 \begin{theo}\label{theo1.1} Assume the Bloch-Beilinson conjecture. 
 Then the $\ch^{nd_+}(S)$-valued generating series (\ref{ch-gen}) is modular for all $n$. 
 \end{theo} 
 
 The idea is to combine the embedding trick with a peculiar property of the Hodge diamond for orthogonal 
Shimura varieties. 
If $U_0$ is a totally positive quadratic space of dimension $4\ell$ over $F$,  the orthogonal sum $\tilde V =U_0+V$ has 
signature $((m+4\ell,2)^{d_+},(m+2+4\ell,0)^{d-d_+})$, and there is a corresponding Shimura variety  $\tilde S$ with an embedding 
$$\rho: S\lra \tilde S$$
of Shimura varieties. The image of the (formal) generating series $\phi_n(\tau,\ph,\tilde S)$ under the pullback
$$\rho^*: \ch^{nd_+}(\tilde S) \lra \ch^{nd_+}(S)$$
is a finite linear combination of products 
\beq\label{pullback-rel}
\theta(\tau,\ph^0)\,\phi_n(\tau,\ph^1,S)
\eeq
where $\theta(\tau,\ph^0)$ is a theta series for 
$\ph^0\in S(U_0(\A_f)^n)$ and $\phi_n(\tau,\ph^1,S)$ is a (formal) $\ch^{nd_+}(S)$-valued generating series 
for $\ph^1\in S(V(\A_f)^n)$.  On the other hand, the results of Vogan and Zuckermann, explained in detail in Section~\ref{section7}, imply that 
\beq\label{VZ-van-1}
H^{2n d_+-1}(\tilde S)=0, \quad \text{for $\ell>nd_+$}.
\eeq
Therefore, if we assume the Bloch-Beilinson conjecture, the series $\phi_n(\tau,\ph,\tilde S)$ is modular, and Theorem~\ref{theo1.1}
follows from the pullback relations (\ref{pullback-rel}).   This pullback argument is analogous to the argument  in \cite{YZZ}, p1159. 
In fact, in our case, the proof of this consequence given in Section~\ref{section8} is quite non-trivial and was provided by Jan Bruinier.  It 
depends on the normality of the Baily-Borel compactification of the Hilbert-Siegel modular variety and some results of Kn\"oller, \cite{knoeller}.
 
\begin{rem} (i) Theorem~\ref{theo1.1} provides support for the conjectured modularity of the Chow group valued generating series, even in the `problematic' $d_+>1$ cases. 
Note that the Bloch-Beilinson conjecture serves as an existence result for the required (but non-evident) relations. \hfb
(ii) One can obviously consider 
analogous unitary Shimura varieties with respect to a CM field over $F$ associated to a Hermitian space of 
signature $((m,1)^{d_+},(m+1,0)^{d-d_+})$. The special cycles occur in codimensions $nd_+$ so that, when $d_+>1$,  the modularity of the Chow group valued 
generating series for such cycles case is again problematic. Unfortunately, there is no evident Hodge diamond argument in this case.\hfb 
\end{rem}

We now give a brief summary of the contents of this paper.  In Section~\ref{section2}, we define the special cycles and the generating series for 
them in classical language.  We also explain how the modularity of the Chow group valued generating series follows from the Bloch-Beilinson 
conjecture together with a vanishing theorem for low odd degree cohomology of orthogonal Shimura varieties. 
In Section~\ref{section3}, we point out a couple of natural questions/problems that arise when $d_+>1$.  Section~\ref{section4} is the core of paper. 
Here, working in classical language, we give formulas for the intersection products on special cycles using the machinery of Fulton \cite{fulton.book}. 
There are several basic ingredients. First, using the Jaffee Lemma, Lemma~\ref{jaffee.lem} and its variant, Proposition~\ref{prop-good-cover}, 
we pass to covers so that the embedding of the special cycles and the components of their intersections are regular embeddings.
In this situation, the intersection product can be expressed in terms of Chern classes and Segre classes of normal cones, Proposition~\ref{prop-normal-cones}.  
These, in turn, can be computed in terms of an excess bundle which is finally related, cf. Proposition~\ref{prop-excess-bundle},  
to the `co-tautological' bundle $\boldCC$, defined in (\ref{bold-CS}).  Thus we obtain a nice formula for the intersection product, 
in classical language, Theorem~\ref{theo4.12}. 
In Section~\ref{section5}, we give the definition of weighted special cycles in ad\`elic language.  These cycles are compatible with pullbacks 
and hence define classes in the Chow group $\ch^{nd_+}(S) := \varinjlim_K \ch^{nd_+}(S_K)$ as $K$ runs over compact open subgroups of $G(\A_f)$, where 
$G = R_{F/\Q}\GSpin(V)$. The product formula (\ref{intro-prod-form}) and Proposition~\ref{prop5.1} for weighted special cycles then follows from the classical version.
A formula for pullbacks of weighted special cycles to Shimura subvarieties is proved in Section~\ref{section6}, Proposition~\ref{prop-6.2}.  This provides the 
basis for the first step in the embedding trick discussed in Section~\ref{section8}, an identity, (\ref{embed-step1}), expressing the pullback of the formal generating 
series for an ambient orthogonal Shimura variety as a product of the formal generating series for $S$ and a Hilbert-Siegel theta function.  
The fact that the modularity of the ambient generating series for a family of such identities implies the modularity of the series for $S$ is proved in 
Section~\ref{sectionFFS}. The vanishing of the low odd degree cohomology of an orthogonal Shimura variety as a consequence of the results of Vogan and Zuckerman, 
is explained in Section~\ref{section7}. Finally, in Section~\ref{section10}, we work out in detail the relation between the weighted special cycles 
as defined in Section~\ref{section5} and an alternative definition analogous to that used in \cite{K.duke}. This relation, which involves a careful 
discussion of the connected component and the structure of the special $0$-cycles arising when $n=m$,  will be useful in certain applications.

\section{Generating series for special cycles: classical version}\label{section2}

In this section, we set up the generating series for algebraic cycles on our orthogonal Shimura variety over a totally real field. 
Here we formulate things in classical language so that the geometric aspects are clearer. An ad\`elic version is described in Section~\ref{section4}. 

Let $F$ be a totally real field of degree $d=|F:\Q|$ and let $\Sigma= \{\s_j\}$ be the set of archimedean embeddings of $F$.
Let $V$, $(\ ,\ )$ be a quadratic space over $F$ with 
$$\sig(V_{j}) = \begin{cases} (m,2) &\text{for $1\le j \le d_+$}\\
\nass
(m+2,0)&\text{for $d_+<j\le d$,}
\end{cases}
$$
where $V_j = V\tt_{F,\s_j}\R$.   We will write $\Sigma_+ = \{ \s_j\mid j\le d_+\}$. 
Let
$$\Dp= \prod_{1\le j \le d_+} \Djp,$$
where $\Djp$ is one component of the space $D^{(j)}$ of oriented negative $2$-planes  in $V_j$.  Thus $\Dp$ is connected and $\dim_\C \Dp = md_+$. 
The space 
$$D = \prod_j D^{(j)}$$ 
has $2^{d_+}$ connected components and will be used in the ad\`elic version 
in Section~\ref{section4},

Let $L\subset V$ be an $O_F$-lattice on which $Q(x) = \frac12(x,x)$ is $O_F$-valued and let 
$$L^\vee = \{ \, x\in V(F)\mid (x,L)\subset \d_F^{-1}\,\}\ \supset L$$
be the dual lattice\footnote{Later, when we consider the Weil representation, this definition of $L^\vee$ will be appropriate when
we use the `standard' additive character $\psi_0$ of $\Q_\A/\Q$ and the character $\psi = \psi_0\circ \tr_{F/\Q}$ for $F_\A/F$.}, 
 where $\d_F^{-1}$ is the inverse different of $F/\Q$. 
Let 
$$ \Gamma_L = \{ \gamma \in \SO(V) \mid \gamma L = L, \gamma\vert{L^\vee/L}= \text{id}\ \},$$
and let $\Gamma \subset \Gamma_L$ be a neat subgroup of finite index which stabilizes the component $\Dp$. In particular, $\Gamma$ is torsion free. 
The quotient 
\beq\label{classical-S}
S = S_\Gamma = \Gamma\back \Dp \ \overset{\pi}{\longleftarrow} \Dp, \qquad \pi = \pi_\Gamma,
\eeq
is then (isomorphic to the set of complex points of) a smooth quasi-projective variety over $\C$ and is projective if $d_+<d$. 
It is a connected Shimura variety with a canonical model over a number field, but we will not need this for the moment. 
Let 
$\ch^i(S)$ be the Chow group of algebraic cycles of codimension $i$ on $S$ modulo rational equivalence
and let 
$$\ch^\bullet(S) = \oplus_{i=0}^{md_+} \ch^i(S)$$
be the Chow ring of $S$. We frequently make the identification $\Pic(S)= \ch^1(S)$, $\mathcal L \mapsto c_1(\mathcal L)$. 

Special cycles are defined as follows.  For a subspace $W\subset V$ which is totally positive definite for $Q$, 
let 
\beq\label{DpW}
\Dp_W = \prod_j \Djp_W, 
\eeq
where
$$
\Djp_W = \{ z_j\in \Djp\mid z_j \subset W^\perp\tt_{F,\s_j}\R\ \}.
$$
In particular, the codimension of $\Dp_W$ in $\Dp$ is $r(W) d_+$ where $r(W) = \dim_F W$, and  
\beq\label{basic-cycle}
Z(W) =Z(W)_\Gamma = \pi_\Gamma(\Dp_W)
\eeq 
is an algebraic cycle of codimension $r(W) d_+$ in $S$.  The corresponding class in $\ch^{r(W) d_+}(S)$ 
will be denoted by $[Z(W)]$. 

On the quadric model 
\beq\label{quadric-model}
D^{(j)} \simeq \{ \ w_j\in (V_j)_\C\mid (w_j,w_j)=0, \ (w_j,\bar w_j)<0\ \}/\C^\times \  \subset \ \mathbb P((V_j)_\C),
\eeq
let $\Cal L_j^\nat$ be the restriction of the tautological line bundle on $\mathbb P((V_j)_\C)$. 
Let $\mathcal L_j = \pr_j^*\mathcal L_j^\nat$ be the pullback of $\mathcal L_j^\nat$ to $D$,  where $\pr_j$ is the projection onto the $j$th factor. 
The restriction of this line bundle to $\Dp$ descends to $S$, where we denote it by the same symbol, and we obtain a class $c_1(\mathcal L_j)\in \text{\rm CH}^1(S)$. 
 Let 
\beq\label{def-cS}
\cbold_S = \prod_{j=1}^{d_+} c_1(\mathcal L_j^\vee) \ \in \text{\rm CH}^{d_+}(S).
\eeq
We will also need the vector bundle, the co-tautological bundle, 
\beq\label{bold-CS}
\boldCC_S = \oplus_j \mathcal L_j^\vee
\eeq
of rank $d_+$. 
The fibers of this bundle are naturally $F$-vector spaces and 
$$\cbold_S = c_{d_+}(\boldCC_S)\cap [S]$$
where $c_{d_+}(\boldCC_S)$ is the top Chern class of $\boldCC_S$.  Here we are using the conventions of Chapter 3 of \cite{fulton.book}.
Later, when we vary $\Gamma$, we will write $\cbold_\Gamma$ and $\boldCC_\Gamma$ to indicate the dependence on $\Gamma$.

For $x\in V^n$, let $W(x)$ be the subspace of $V$ spanned by the components of $x$ and let $r(x)= \dim W(x)$. 
Let 
\beq\label{shifted-cycle}
[Z(x)] = \begin{cases}  [Z(W(x))] \cdot \cbold_S^{n-r(x)}&\text{if $W(x)$ is positive definite,}\\
\nass
0&\text{otherwise.}
\end{cases}
\eeq
Thus, $[Z(x)]\in \ch^{n d_+}(S)$. For example, $[Z(0)] = \cbold_S^n$. 
When we vary $\Gamma$, we will write $[Z(x)_\Gamma]$. 

The following equivariance property will be useful later. If $\eta\in \SO(V)(F)$, under the natural isomorphism, 
\beq\label{brak-eta} 
[\eta]: S_\Gamma \isoarrow S_{\eta \Gamma \eta^{-1}}, \qquad z\mapsto \eta z,
\eeq
\beq\label{eta-trans} 
[\eta]_* Z(W)_\Gamma = Z(\eta W)_{\eta \Gamma \eta^{-1}} \qquad\text{and}\qquad [\eta]_* [Z(x)]_\Gamma = [Z(\eta x)]_{\eta \Gamma \eta^{-1}}.
\eeq

For the lattice $L$, let $\mathbb S(L) = \C[(L^\vee/L)^n]$ be the group algebra of $(L^\vee/L)^n$.
Define the generating series 
$$\phi_n(\tau, S) = \sum_{\mu\in (L^\vee/L)^n} \sum_{\substack{ x\in \mu+ L^n\\ \snass \mod \Gamma}}  [Z(x)]\, \qq^{Q(x)}\cdot \ebold_\mu\quad \in \ 
\ch^{nd_+}(S)\tt \mathbb S(L)[[\qq]],$$
where $\{\ebold_\mu\}$ is the coset basis for $\mathbb S(L)$ and, for $T\in \Sym_n(F)$ and $\tau\in (\H_n)^d$, $\qq^T$ is given by (\ref{def-qT}).

There is a unitary representation $\rho_L$ of $\Gamma'$ on the space\footnote{As usual, $\mathbb S(L)$ can be identified with a subspace of 
the Schwartz space $S(V(\A_f)^n)$ of finite adeles over $F$ of $V^n$ and the representation $\rho_L$ has a natural construction in this language.}  $\mathbb S(L)$
where $\Gamma' = \Sp_n(O_F)$, if $m$ is even, s and is a $2$-fold central extension of this group, if $m$ is odd. 

The expectation is that $\phi_n(\tau, S)$ is the $\qq$-expansion of a Hilbert-Siegel modular form of genus $n$ and parallel weight 
$\kappa=\frac{m}2+1$. This means that, for any linear 
functional $\l: \text{\rm CH}^n(S) \lra \C$, the series 
$$\phi_n(\tau, S,\l) = \sum_{\mu\in (L^\vee/L)^n} \sum_{\substack{ x\in \mu+ L^n\\ \snass \mod \Gamma}}  \l([Z(x)])\, \qq^{Q(x)}\cdot \ebold_\mu,$$
with coefficients in $\mathbb S(L)$,  is termwise absolutely convergent and that the resulting analytic function on $\H_n^d$  satisfies
$$\phi_n(\gamma(\tau), S, \l) = \prod_j \det(c\tau_j+d)^{\kappa} \,\rho_L(\gamma) \,\phi_n(\tau, S,\l), $$
for all $\gamma\in \Gamma'$. 

As motivation, one has the fact that the image of $\phi_n(\tau,S)$ under the cycle class map 
\beq\label{cycle-class}
\text{\rm cl}^{nd_+}: \ch^{nd_+}(S) \lra H^{2nd_+}(S)
\eeq
is  
the $\qq$-expansion of a Hilbert-Siegel modular form, 
$$\phi_n(\tau, S,\text{\rm cl}) = \sum_{\mu\in (L^\vee/L)^n} \sum_{\substack{ x\in \mu+ L^n\\ \snass \mod \Gamma}}  
\text{\rm cl}([Z(x)])\,\qq^{Q(x)}\cdot \ebold_\mu\quad \in H^{2nd_+}(S)\tt\mathbb S(L), $$
by the results of \cite{KM1}, \cite{KM2} and \cite{KM3}. 

Of course, as observed in \cite{YZZ}, when the cycle class map \ref{cycle-class} is injective, the modularity of $\phi_n(\tau, S,\text{\rm cl})$ implies that of $\phi_n(\tau, S)$. 
Let $\text{\rm CH}^{N}(S)^0 = \ker(\text{\rm cl}^{N})$ be the subgroup of $\text{\rm CH}^{N}(S)$ of cohomologically trivial cycles
and let 
$$\text{AJ}_{N}: \text{\rm CH}^{N}(S)^0 \lra J^{N}(S)$$
be the Abel-Jacobi map to the $N$th intermediate Jacobian of $S$. The Bloch-Beilinson Conjecture asserts that the map  $\text{AJ}N{N}$ is injective up to torsion. 
Here recall that $S$ is defined over a number field. 
On the other hand, we will show in Section~\ref{section7} that $J^N(S) = 0$ for $2N-1 < m-\left[\frac{m}2\right]$. 
This proves the following. 
\begin{prop}  Assume the Bloch-Beilinson conjecture.  Then the $\text{\rm CH}^{n d_+}(S)$-valued series $\phi_n(\tau, S)$ is modular for 
$$nd_+ < \begin{cases}  \frac{m+2}4 &\text{for $m$ even,}\\
\nass
\frac{m+3}4 &\text{for $m$ odd.}
\end{cases}
$$ 
\end{prop}

Now we apply the `embedding trick', as described in \cite{YZZ}, p.1159. 
Our description here is imprecise; a precise version using the the ad\`elic formulation of the generating series of Section~\ref{section4}
will be given in Section~\ref{section8}.  Let $U_0$ be totally positive quadratic space over $F$ of dimension $4\ell$,  and let $\widetilde{V}= U_0\oplus V$. 
Suppose that $L_{U_0}$ 
is an even integral lattice in $U_0$ and let $\widetilde{L}= L_{U_0}\oplus L$.   Note that 
\beq\label{sig-embed}
\sig(\widetilde{V}) =( (m+4\ell,2)^{d_+}, (m+2+4\ell,0)^{d-d_+})
\eeq 
in the obvious notation. Let $\widetilde{D}^+ = \prod_{j=1}^{d_+} \widetilde{D}^{(j),+}$ be the associated symmetric space and take a neat subgroup $\widetilde{\Gamma}$ of finite index 
in the group $\Gamma_{\widetilde{L}}$. We suppose\footnote{This will be handled in a better way in Section~\ref{section7}} that $\widetilde{\Gamma}\cap \SO(V) = \Gamma$ and 
thus have an embedding 
$$j: S \lra \widetilde{S} = \widetilde{\Gamma}\back \widetilde{D}.$$
If $\ell$ is sufficiently large, e.g., $\ell>nd_+$ will always work, and assuming the Bloch-Beilinson conjecture, the series 
$\phi_n(\tau,\widetilde{S})$ is modular of weight $\kappa + 2 \ell$, valued in $\text{\rm CH}^{n d_+}(\widetilde{S})\tt \mathbb S(\widetilde L)$.  
On the other hand, the pullback of this series under $j$, should be expressible as a finite linear combinations of products of theta series associated to $L_{U_0}$, of weight $2\ell$,  and 
components of generating series $\phi_n(\tau,S)$. Using a suitable cancellation property, the modularity of $\phi_n(\tau,S)$ will follow for all $n$!
Thus, up to several compatibilities and the pullback formula and cancellation properties which will be carefully formulated in Section~\ref{section7}, we have the following. 
\begin{theo}\label{MAIN-prop}  Assume the Bloch-Beilinson conjecture.  
Then the $\text{\rm CH}^{n d_+}(S)$-valued series $\phi_n(\tau, S)$ is modular for all $n$. 
\end{theo}

Note that the range of vanishing of odd Betti numbers given by Corollary~\ref{cor-better-van} plays a crucial role here.

\section{Problematic examples}\label{section3}

In this section, we make some observations about relations among special cycles. 
The key point is that the codimensions of special cycles are multiples of $d_+$.   Thus, when $d_+=1$, there are special cycles 
defined in each codimension and it is reasonable to imagine that relations among cycles of codimension $n$ arise from meromorphic functions on special 
cycles of codimension $n-1$. When $d_+>1$, this is no longer the case and there is no evident source of such relations, whereas the modularity of the
generating series for such cycles implies that such relations must exist in abundance. Thus there is an essential difference between the cases $d_+=1$ and $d_+>1$. 

In the case $d_+=1$ and $F=\Q$, one might imagine that the meromorphic functions on special cycles of codimension $n-1$ giving rise to relations among special cycles of codimension $n$ 
are those constructed by Borcherds on such 
$m-n+1$-dimensional orthogonal Shimura subvarieties. In fact, the Zhang, Bruinier-Westerholt-Raum proof of modularity does not proceed in this way and this suggests the following problem. \hfb
{\bf Problem 1.} In the case $F=\Q$, what are the relations among the special cycles 
$$Z(T) = \sum_{\substack{x\in (L^\vee)^n\\ \snass Q(x) = T\\ \snass \mod \Gamma}} Z(x)$$
implied by the modularity of the generating series. Can these be described in terms of Borcherds forms on $Z(y)$'s where $y\in (L^\vee)^{n-1}$?

Now suppose that $d_+>1$.   
 
{\bf Example.}   
Suppose that $F$ is a real quadratic field and let $B$ be a division quaternion algebra over $F$ that is split at the archimedean places $\s_1$ and $\s_2$. 
We also suppose that $B$ is not a base change of an indefinite quaternion algebra over $\Q$, e.g., that $B_v$ is a division algebra for some non-archimedean place over a 
rational prime $p$ that is not split in $F$.  Let $V$ be the subspace of elements $x\in B$ with $\tr(x)=0$, where $\tr: B\ra F$ is the reduced trace, and let $Q(x)= \nu(x)$ 
be the reduced norm of $x$. Then 
$\sig(V) = ((1,2),(1,2))$
so that $m=1$ and $d_+=2$ in our notation above.  Choosing $L$ and $\Gamma$, we obtain a smooth projective surface $S$ with a large supply of $0$-cycles $Z(x)$
defined by vectors $x\in V$ with $Q(x)\gg0$.   
The associated generating series is
\beq\label{0-cycle-gen}
\phi_1(\tau, S) = \sum_{\mu\in (L^\vee/L)^n} \sum_{\substack{ x\in \mu+ L\\ \snass \mod \Gamma}}  [Z(x)]\cdot \qq^{Q(x)}\cdot \ebold_\mu\quad \in \ 
\text{\rm CH}^{2}(S)\tt \mathbb S(L)[[\qq]].
\eeq
The image of this series under the cycle class map 
$$\text{\rm cl}^2: \text{\rm CH}^2(S) \lra H^4(S,\C) =\C$$
is a Hilbert modular form 
$$\phi_1(\tau, S,\text{\rm cl}) = \sum_{\mu\in (L^\vee/L)^n} \sum_{\substack{ x\in \mu+ L\\ \snass \mod \Gamma}}  \deg Z(x)\cdot \qq^{Q(x)}\cdot \ebold_\mu$$
of parallel weight $\frac32$ valued in $\mathbb S(L)$. 
The modularity of $\phi_1(\tau, S)$ must entail a large number of relations among the $0$-cycles $Z(x)$, but such relations would arise from collections $\sum_j (C_j,f_j)$ 
where $C_j$ is a curve on $S$ and $f_j$ is a meromorphic function on $C_j$.  But there are no evident curves on $S$!
Nonetheless, Proposition~\ref{MAIN-prop} asserts the modularity of $\phi_1(\tau, S)$ assuming the Bloch-Beilinson conjecture. 

{\bf Problem 2.}  Find explicit relations among the $0$-cycles $Z(x)$ on $S$. 

{\bf Problem 3.} Use them to prove modularity of (\ref{0-cycle-gen}).

\section{Some intersection theory}\label{section4}

In this section, we record some results about the geometry and intersections of the special cycles. 
We begin with the classical version of Section~\ref{section2} and will pass back and forth, via GAGA, between topological 
and algebraic geometric arguments. In particular, since we will be working with projective varieties over $\C$, we follow the treatment of 
intersection theory given in Fulton, \cite{fulton.book}.  Thus we sometimes write $A_k(X)$ for the group of $k$-cycles modulo rational equivalence
and note that $A_k(X) = \ch^{n-k}(X)$ if $X$ is smooth of dimension $n$. 

\begin{rem}  The intersection of (weighted ad\`elic) special cycles was considered in \cite{YZZ} in the case $d_+=1$. 
It may be that their formulation can be extended to the case $d_+>1$, but we felt that the more classical approach 
given here with complete proofs provides a better insight into the geometry.   
\end{rem}

As in Section~\ref{section2}, we suppose that $\Gamma$ is a neat subgroup of $\Gamma_L$ preserving the component $\Dp$
and is, in particular, torsion free. 

\subsection{Some preliminary results.}

If $\Gamma'\subset \Gamma$ is a subgroup of finite index, the map $\pr: S_{\Gamma'}\ra S_\Gamma$ is, topologically, a covering map and hence, algebraically, is 
finite \'etale of degree $|\Gamma:\Gamma'|$. 
The map 
$$\pr^*: \ch^\bullet(S_\Gamma) \lra \ch^\bullet(S_{\Gamma'})$$ 
is a ring homomorphism. Since, for $\a\in A_k(S_\Gamma)$, 
$$\pr_*\pr^*(\a) =  |\Gamma:\Gamma'|\,\a,$$ 
$\pr^*$ is injective. 
In particular,
$$\pr_*(\pr^*(\a)\cdot \pr^*(\b)) =   \pr_*(\pr^*(\a\cdot\b)) = |\Gamma:\Gamma'|\, \a\cdot \b,$$
so that identities involving products of elements of  $\ch^\bullet(S_\Gamma)$ can be checked on their pullbacks. 
Also note that 
$$\pr_*(\cbold_{\Gamma'}) = \cbold_\Gamma.$$

\begin{rem} For a totally positive subspace $U$ in $V$, the cycle $D_U$ in $D$ is a holomorphic and totally geodesic submanifold and 
$\Gamma$ acts on $D$ by holomorphic isometries.  
Thus, if the restriction of the (topological covering) map $\pi_\Gamma:\Dp \ra \Gamma\back \Dp$ to $\Dp_U$ is injective, the image is a totally geodesic holomorphic 
submanifold and the inclusion  of this image in $S_\Gamma$ is (algebraically) a regular embedding. 
Similarly, for totally positive subspaces $W\subset U\subset V$, if the restriction of $\pi_\Gamma$ to $\Dp_U$ and to $\Dp_W$ is injective, 
then the image of $\Dp_W$ is a totally geodesic submanifold of the image of $\Dp_U$ and the inclusion is a regular embedding. 
\end{rem}

The following result and its variants will be useful.  It holds in a much more general context, c.f., \cite{RS}.  For convenience, we include the proof. 
\begin{lem}\label{jaffee.lem}
 \cite{millson.rag}, \cite{RS}.  Let $U$ be a subspace of $V$ which is totally positive definite for $Q$, and let $\s_U$ be the isometry of $V$ 
with $\s_U\vert_U = -1$ and $\s_U\vert_{U^\perp}=+1$.  \hfb 
(i)  Let $\tilde\Gamma_U$ be the 
centralizer of $\s_U$ in $\Gamma$, i.e., the stabilizer in $\Gamma$ of the subspace $U$.   Let  $\Gamma_U$ be the subgroup of $\tilde\Gamma_U$ 
whose elements act trivially on $U$.  Then, since $\Gamma$ is neat, $\tilde\Gamma_U = \Gamma_U$. \hfb
(ii)  (Jaffee Lemma) Suppose that $\s_U \Gamma\s_U = \Gamma$.  
Then the map 
$$\tilde\Gamma_U\back \Dp_U \lra \Gamma\back \Dp$$
is injective.\hfb 
In particular,  this implies that $Z(U)_\Gamma= \Gamma_U\back \Dp_U$ is a submanifold of $S_\Gamma$ and 
the map 
$$f:Z(U)_\Gamma \lra S_\Gamma$$
is a regular embedding of codimension $r(U) d_+$. 
\end{lem} 
\begin{proof} 
To prove (i), note that,  
since $\Gamma$ is neat, so are $\tilde \Gamma_U$ and its image in $O(U)$.  Since $U$ is totally positive definite and the image 
of $\tilde \Gamma_U$ in $O(U)(\R)$ is discrete,  this image must be torsion and hence trivial. Thus $\gamma \in \Gamma_U$, as required. 
To prove (ii), suppose that $z$ and $z'\in \Dp_U$ and that $\gamma z' = z$ for some $\gamma\in \Gamma$. 
Then, since $\s_U$ fixes $\Dp_U$ pointwise,  $\s_U\gamma \s_U z' = z$ as well. Since $\Gamma$ is torsion free and hence acts without fixed points on $\Dp$, 
we must have $\gamma^{-1} \s_U\gamma\s_U= 1$ and so $\s_U\gamma\s_U = \gamma\in \tilde\Gamma_U$, as required. 
Combining this with (i) gives the last statement. 
\end{proof}

\begin{rem}\label{rem8.3} (i) If a lattice $L$ satisfies $L = U\cap L + U^\perp\cap L$, then $\s_U \Gamma_L \s_U = \Gamma_L$. \hfb 
(ii) The condition of the lemma will always hold after passing to a subgroup of finite index. 
For example, for a totally positive subspace $U$ of $V$, let 
$$\Gamma' = \Gamma\cap \s_U\Gamma \s_U.$$
Then $\s_U\Gamma'\s_U = \Gamma'$ and  we have
$$
\xymatrix{ Z(U)_{\Gamma'}=\Gamma'_U\back \Dp_U\ar[r]^>>>>{f'}\ar@<25pt>[d]& \Gamma'\back \Dp\ar[d]\\
Z(U)_\Gamma=\Gamma_U\back \Dp_U\ar[r]^>>>>>f&\Gamma\back \Dp, 
}
$$
where $f'$ is a regular embedding. \hfb 
(iii) If $f: Z(U)_\Gamma \ra S_\Gamma$ is a regular embedding and $\Gamma'\subset \Gamma$ has finite index, then \hfb
$f': Z(U)_{\Gamma'}\ra S_{\Gamma'}$ is also a regular embedding. \hfb
\end{rem}

\subsection{Intersections}
Suppose that $U_1$ and $U_2$ are totally positive subspaces of $V$ with associated classes 
$$[Z(U_i)_\Gamma]\ \in\  \ch^{r_i d_+}(S_\Gamma), \qquad r_i = r(U_i).$$ 
We want to compute the product 
$$[Z(U_1)_\Gamma]\cdot [Z(U_2)_\Gamma]\ \in \ \ch^{(r_1+r_2)d_+}(S_\Gamma).$$
The following is the analogue of Proposition~2.2 in \cite{YZZ}. 
\begin{prop}\label{prop8.4} (i) As a set, the fiber product 
\beq\label{raw-inter-0}
\xymatrix{|I(U_1,U_2)_\Gamma| \ar[d]\ar[r]&|Z(U_2)_\Gamma|\ar[d]\\
|Z(U_1)_\Gamma|\ar[r]& |S_\Gamma|
}
\eeq 
is given by 
\beq\label{raw-inter-02}
|I(U_1,U_2)_\Gamma| = \bigcup_{W} |Z(W)_\Gamma|, 
\eeq
where $W$ runs over the set 
\beq\label{W-orbits-0}
\Gamma\back \{\ W= \gamma_1 U_1+\gamma_2 U_2\mid \gamma_1, \gamma_2\in \Gamma\ \}.
\eeq
(ii) As a scheme, the fiber product 
\beq\label{raw-inter-1}
\xymatrix{I(U_1,U_2)_\Gamma \ar[d]_g\ar[r]&Z(U_2)_\Gamma\ar[d]\\
Z(U_1)_\Gamma\ar[r]& S_\Gamma
}
\eeq 
is given by 
\beq\label{raw-inter-2}
I(U_1,U_2)_\Gamma = \bigcup_{\gammabold} Z(W_{\gammabold})_\Gamma, 
\eeq
where the subspace $W_{\gammabold}$ is given by $\gamma_1U_1+\gamma_2 U_2$ 
as the pair $\gammabold=(\gamma_1,\gamma_2)$ runs over representatives for the $\Gamma$-orbits in the set 
\beq\label{better-index} 
\text{\rm Inc}(U_1,U_2)_\Gamma:= \Gamma/\Gamma_{U_1} \times \Gamma/\Gamma_{U_2}.
\eeq
\end{prop}

\begin{rem} (i) Note that the map 
$$\Gamma\back \big(\ \Gamma/\Gamma_{U_1} \times \Gamma/\Gamma_{U_2}\ \big) \lra \Gamma\back \{\ W= \gamma_1 U_1+\gamma_2 U_2\mid \gamma_1, \gamma_2\in \Gamma\ \}$$
has finite fibers which give rise to multiplicities in the fiber product. For example, suppose that $\dim_F U_2<\dim_F U_1$ and that $\gamma_0 U_2 \subset U_1$ and $\gamma_0' U_2\subset U_1$,  
for some $\gamma_0$ and $\gamma_0'\in \Gamma$.  Then the pairs $(1,\gamma_0)$ and $(1,\gamma_0')$ both map to the $\Gamma$-orbit of $U_1$ in (\ref{W-orbits-0}). 
But the subspaces  $\gamma_0 U_2$ and $\gamma'_0 U_2$ of $U_1$ can be distinct and hence, since elements of $\Gamma_{U_1}$ act trivially on $U_1$,  the double cosets
$\Gamma_{U_1}\gamma_0\Gamma_{U_2}$ and $\Gamma_{U_1}\gamma_0'\Gamma_{U_2}$ can be distinct as well. \hfb 
(ii)  We will often use representatives of the form $(1,\gammabold)$ for orbits in (\ref{better-index}), with a slight abuse of notation. 
\end{rem}

By Proposition~6.1 (a) of \cite{fulton.book} we have the following. 
\begin{prop}\label{prop-normal-cones}
Suppose that $i_1: Z(U_1)_\Gamma \ra S$ is a regular embedding. Let $N = g^*(N_{Z(U_1)_\Gamma}(S_\Gamma))$ be the pullback of the normal bundle, 
where $g$ is as in (\ref{raw-inter-1}), and let $I = I(U_1,U_2)_\Gamma$.  \hfb 
(i) Then 
$$Z(U_1)_\Gamma\cdot Z(U_2)_\Gamma = \{\, c(N) \cap s(I, Z(U_2)_\Gamma)\,\}_{\kappa},$$
where 
$$\kappa = \dim Z(U_1)+\dim Z(U_2)- \dim S,$$
$c(N)$ is the total Chern class of $N$, and 
$$s(I, Z(U_2)_\Gamma) = s(C)$$
is the Segre class of the normal cone
$$C= C_{I} (Z(U_2)_\Gamma)$$
of $I$ in $Z(U_2)_\Gamma$. \hfb 
\end{prop} 
We will see below that this expression can be written as a sum of contributions from the various $Z(W_{\gammabold})_\Gamma$. 
To evaluate these contributions we will use the next result, whose proof we omit, compare 
Section~6.1 of Fulton \cite{fulton.book}, in particular, Example 6.1.7.
\begin{prop}\label{inter-prop-1}  For totally positive subspaces $U_1$ and $U_2$ of $V$ and $W=U_1+U_2$, suppose that 
all of the morphisms in the diagram\footnote{This need not be the fiber product.}
\beq\label{square-good}
\xymatrix{Z(W)_{\Gamma}\ar[r]\ar[d]_q&Z(U_2)_{\Gamma}\ar[d]\\
Z(U_1)_{\Gamma}\ar[r]&S_{\Gamma}
}
\eeq
are regular embeddings. 
Let 
$$N = q^*N_{Z(U_1)_{\Gamma}} S_{\Gamma}$$ 
be the pullback to $Z(W)_\Gamma$ of the normal bundle to $Z(U_1)_{\Gamma}$ in $S_{\Gamma}$ and let 
$$N' = N_{Z(W)_{\Gamma}}Z(U_2)_{\Gamma}$$ 
be the normal bundle to $Z(W)_{\Gamma}$  in $Z(U_2)_{\Gamma}$. 
Note that $N'$ is a sub-bundle of $N$ and that these bundles have ranks $(r(W)-r(U_2))d_+$ and  $r(U_1)d_+$ respectively. The 
excess bundle $E = N/N'$ has rank 
$$e=(\,r(U_1)+r(U_2)-r(W)\,)d_+.$$ 
Then 
$$
\{\, c(N) \cap s(Z(W)_\Gamma, Z(U_2)_\Gamma)\,\}_{\kappa}
=c_e(E)\cap [Z(W)_\Gamma]\ \  \in  A_{k-e}(Z(W)_\Gamma),
$$
where $c_e(E)$ is the top Chern class of $E$ and $k = \dim \Dp_W = (m-r(W))d_+$. Pushing this forward to $S_\Gamma$ yields a class 
\beq\label{fund-contrib}
c_e(E)\cap [Z(W)_\Gamma]\  \in  A_{k-e}(S_\Gamma) = \ch^{(r_1+r_2)d_+}(S_\Gamma).
\eeq
\end{prop} 

\subsection{Excess bundles} 
In this section  we compute the excess bundle in the situation of Proposition~\ref{inter-prop-1}, working with complex manifolds. 

For a point $z_j\in \Djp$, we have a canonical identification of the tangent space\footnote{
Here we note that if $e_1$ and $e_2$ is a properly oriented orthogonal basis for $z_j$ with $(e_1,e_1)=(e_2,e_2) = -1$, then 
the complex structure $J_{z_j}$ on $z_j$ given by $J_{z_j} e_1 = -e_2$, $J_{z_j} e_2 = e_1$ induces a complex structure on $T_{z_j}(\Djp)$; it depends only on the orientation of $z_j$. 
The map (\ref{quadric-model}) sending $z_j$ to the $+i$-eigenspace of $J_{z_j}$ in $(z_j)_\C$ is holomorphic. }
$$T_{z_j}(\Djp) = \Hom(z_j,z_j^\perp),$$
and hence 
$$T_z(\Dp) = \bigoplus_{j=1}^{d_+} \Hom(z_j,z_j^\perp) = \Hom_F(z,z^\perp).$$
Here we are working with subspaces of $V\tt_\Q \R$, and the idempotents in $F\tt_\Q\R$ give the direct sum decomposition of the space of $F$-linear maps. 
As a slight abuse of notation, we are writing $\Hom_F$ for the space of $F_\R$-linear maps.  
Similarly, if $z\in \Dp_W$, we have tangent spaces
\begin{align*}
T_z(\Dp_{U_i}) &= \Hom_F(z,z^\perp\cap U_i^\perp), \quad i=1, 2, \\
\noalign{and}
T_z(\Dp_W) &= \Hom_F(z,z^\perp\cap W^\perp).
\end{align*}
Since $z\in \Dp_{U_1}$, we have an orthogonal decomposition
$$V_\R = U_{1,\R} + z+ (z^\perp \cap U_{1,\R}^\perp).$$
Thus, the fiber at $z$ of the normal bundle to $\Dp_{U_1}$ in $\Dp$ is given by 
\beq\label{normal-bund-1}
N_{\Dp_{U_1}}(\Dp)_z \simeq \Hom_F(z,z^\perp/(z^\perp \cap U_{1}^\perp)) \simeq \Hom_F(z,U_{1,\R}).
\eeq

From the construction of the line bundles $\mathcal L_j$ and the vector bundle $\boldCC_\Gamma$ of Section~\ref{section2}, we see that 
the vector space $(V_j)_\C$ inherits an $F$-vector space structure from $V$, and hence the fibers of each $\mathcal L_j$ and 
of $\boldCC_\Gamma$ are naturally $F$-vector spaces. Thus (\ref{normal-bund-1}) yields the following result. 
\begin{lem}\label{lem-normal-bund} Suppose that the map $Z(U_1)_\Gamma\ra S_\Gamma$ is a regular embedding. 
Then 
$$N_{Z(U_1)_\Gamma}(S_\Gamma) \simeq (\boldCC_\Gamma \tt_F U_1)\vert_{Z(U_1)_\Gamma}.$$
\end{lem}
Next consider $z\in \Dp_W\subset \Dp_{U_2}$, so that the fiber of the normal bundle to $\Dp_W = \Dp_{U_1+U_2}$ in $\Dp_{U_2}$ at $z$ is 
\beq\label{normal-2}
\Hom_F(z, (z^\perp\cap U_2^\perp)/(z^\perp\cap W^\perp)) = \Hom_F(z, (z^\perp\cap U_2^\perp)/(z^\perp\cap U_1^\perp\cap U_2^\perp)),
\eeq
where $W^\perp = U_1^\perp\cap U_2^\perp$. We have the diagram
\beq\label{normal-3}
\xymatrix{z^\perp\cap U_2^\perp \ar[r]\ar[d]& z^\perp\ar[d]\\
(z^\perp\cap U_2^\perp)/(z^\perp\cap W^\perp)\ar[r]^{\qquad j} &z^\perp/(z^\perp\cap U_1^\perp)
}
\eeq
where $j$ is injective. This exhibits $N_{\Dp_W}(\Dp_{U_2})$ as a sub-bundle of $N_{\Dp_{U_1}}(\Dp)$, since 
$$N_{\Dp_W}(\Dp_{U_2})_z = \Hom_F(z,(z^\perp\cap U_2^\perp)/(z^\perp\cap W^\perp))\hookrightarrow \Hom_F(z,z^\perp/(z^\perp\cap U_1^\perp)) = N_{\Dp_{U_1}}(\Dp)_z.$$
The fiber of the excess bundle at $z$ is given by $F$-linear homomorphisms from $z$ into the cokernel of $j$, 
$$E_z = \Hom_F(z, z^\perp/(z^\perp\cap U_1^\perp + z^\perp\cap U_2^\perp)).$$ 
But, in fact, we have a nicer expression. 
\begin{lem}
$$E_z = \Hom_{F}(z, (U_1\cap U_2)_\R).$$
\end{lem} 
\begin{proof}  Recall that $\Hom_F$ is the space of $F_\R$-linear maps. Since $z\in D_W = D_{U_1}\cap D_{U_2}$, we have an inclusion $(U_1\cap U_2)_\R \ra z^\perp$. 
If 
$$x\in (U_1\cap U_2)_\R \cap (z^\perp\cap U_1^\perp + z^\perp\cap U_2^\perp),$$
write $x = w_1+w_2$ with $w_i\in U_i^\perp\cap z^\perp$. Then $(x,x) = (x,w_1)+(x,w_2)=0,$
so that $x=0$. But 
$$\dim_F U_1\cap U_2 = e = \dim_\R z^\perp/(z^\perp\cap U_1^\perp + z^\perp\cap U_2^\perp),$$
so the inclusion gives an isomorphism
$$(U_1\cap U_2)_\R \isoarrow z^\perp/(z^\perp\cap U_1^\perp + z^\perp\cap U_2^\perp).$$
\end{proof}
Thus we have the following nice expression for the excess bundle. 
\begin{prop}\label{prop-excess-bundle}  In the situation of Proposition~\ref{inter-prop-1}, the excess bundle is given by 
$$E \simeq \big(\,\boldCC_\Gamma \tt_F (U_1\cap U_2)\,\big) \vert_{Z(W)_\Gamma}.$$
\end{prop} 

\begin{cor}\label{cor-great} In the situation of Proposition~\ref{inter-prop-1}, 
$$c_e(E)\cap [Z(W)_\Gamma]= \cbold_\Gamma^{r_1+r_2-r(W)}\cap [Z(W)_\Gamma]\  \in  \ch^{(r_1+r_2)d_+}(S_\Gamma).$$
\end{cor} 

\subsection{Passing to covers}

To compute $Z(U_1)_\Gamma\cdot Z(U_2)_\Gamma$, we pass to a cover where the geometry becomes nice so that  Corollary~\ref{cor-great} 
can be applied. 

If $\Gamma'$ has finite index in $\Gamma$ and $\pr_{\Gamma'}:S_{\Gamma'} \ra S_\Gamma$ is the projection, then 
\begin{align*}
(\pr_{\Gamma'})_*([Z(U)_{\Gamma'}]) &= [Z(U)_\Gamma], \\
\noalign{and}
\pr_{\Gamma'}^*([Z(U)_\Gamma]) &=  \sum_{\gamma\in \Gamma'\back \Gamma/\Gamma_U} [Z(\gamma U)]_{\Gamma'}.\\
\noalign{Moreover, }
\pr_{\Gamma'}^*(I(U_1,U_2)) &= \bigcup_{\gammabold} Z(W_{\gammabold})_{\Gamma'},
\end{align*}
where $W_{\gammabold} = \gamma_1 U_1+\gamma_2 U_2$ as $\gammabold=(\gamma_1,\gamma_2)$ runs over a set of orbit representatives for
\beq\label{index-set}
\Gamma'\back \big(\ \Gamma/\Gamma_{U_1}\times \Gamma/\Gamma_{U_2}\ \big).
\eeq
Here note that we are simply passing from the $\Gamma$-orbits in Proposition~\ref{prop8.4} to $\Gamma'$-orbits.            

\begin{prop}\label{prop-good-cover} For totally positive subspaces $U_1$ and $U_2$ of $V$, choose a set of representatives 
$\{ \gammabold_j\}$    
for the $\Gamma$-orbits in (\ref{better-index}) and let $W_j = W_{\gammabold_j}$. Then there exists a subgroup $\Gamma'$ of finite index in $\Gamma$ such that, 
for all $\gamma\in \Gamma$, 
all of the morphisms 
$$Z(\gamma U_1)_{\Gamma'}\lra S_{\Gamma'}, \quad Z(\gamma U_2)_{\Gamma'}\lra S_{\Gamma'}, \quad\text{and}\quad\ Z(\gamma W_{j})_{\Gamma'} \lra S_{\Gamma'},$$
are injective and hence are regular embeddings. 
Moreover, for $W_{j}= \gamma_{1,j}U_1+\gamma_{2,j}U_2$ and for all $\gamma\in \Gamma$, the morphism 
$$Z(\gamma W_{j})_{\Gamma'} \lra Z(\gamma \gamma_{2,j}U_2)_{\Gamma'}$$
is injective and hence is a regular embedding as well. 
\end{prop} 
\begin{proof}  Suppose that $\Gamma'$ is a normal subgroup of finite index in $\Gamma$. 
For any $\gamma\in \Gamma$ and totally positive subspace $U$ of $V$,  $\s_{\gamma U} = \gamma \s_U\gamma^{-1}$. Thus, if 
$\Gamma'\subset \Gamma\cap \s_U\Gamma\s_U$,  we have  
$$\Gamma' \ \subset\ \Gamma \cap \s_{\gamma U} \Gamma \s_{\gamma U} = \gamma (\Gamma \cap \s_{U}\Gamma \s_{U})\gamma^{-1}$$
as well. 
By Lemma~\ref{jaffee.lem} and (iii) of Remark~\ref{rem8.3}, it follows that $Z(\gamma U)_{\Gamma'}\ra S_{\Gamma'}$ is a regular embedding.
Thus any normal subgroup $\Gamma'$  of $\Gamma$ such that 
$$\Gamma' \ \subset \ \Gamma \cap (\s_{U_1}\Gamma\s_{U_1} )\cap (\s_{U_2}\Gamma\s_{U_2})\cap \bigcap_{j} (\s_{W_j}\Gamma\s_{W_j})$$
has the required properties. Note that, for $W_j = \gamma_{1,j}U_1+\gamma_{2,j}U_2$, there is a factorization 
$$Z(\gamma W_j)_{\Gamma'} \lra Z(\gamma \gamma_{2,j}U_2)_{\Gamma'} \lra S_{\Gamma'},$$
so that the map 
$$Z(\gamma W_j)_{\Gamma'} \lra Z(\gamma \gamma_{2,j}U_2)_{\Gamma'}$$
is injective and hence is a regular embedding. 
\end{proof}

For totally positive subspaces $U_1$ and $U_2$ of $V$, take a subgroup $\Gamma'\subset \Gamma$ of finite index as in Proposition~\ref{prop-good-cover}, so that we have the 
the diagram
\beq\label{raw-inter-good}
\xymatrix{I(U_1,U_2)_{\Gamma'} \ar[d]_g\ar[r]^j&Z(U_2)_{\Gamma'}\ar[d]^{i_2}\\
Z(U_1)_{\Gamma'}\ar[r]_{i_1}& S_{\Gamma'}
}
\eeq 
is given by 
\beq\label{raw-inter-3}
I(U_1,U_2)_{\Gamma'} = \bigcup_{\gammabold} Z(W_{\gammabold})_{\Gamma'}. 
\eeq
where $\gammabold$ runs over a set of orbit representatives for 
\beq\label{cone-comp-index}
\Gamma'\back \big(\ \Gamma/\Gamma_{U_1}\times \Gamma/\Gamma_{U_2}\ \big).
\eeq
Moreover, $i_1$ and $i_2$ are regular embeddings and $I(U_1,U_2)_{\Gamma'}$ is a union of smooth subvarieties $Z(W_\gammabold)_{\Gamma'}$ of $S_{\Gamma'}$ 
intersecting cleanly. In this situation, we have the following result, whose proof we omit.
\begin{lem}\label{lem-cone-comp}  Let $C' = C_{I(U_1,U_2)_{\Gamma'}}(Z(U_2)_{\Gamma'})$ be the 
normal cone of $I(U_1,U_2)_{\Gamma'}$ in $Z(U_2)_{\Gamma'}$. 
Then the decomposition of $C'$ into irreducible components is given by 
$$C' = \bigcup_{\gammabold} C'_\gammabold, \qquad  C'_\gammabold = N_{Z(W_\gammabold)_{\Gamma'}}(Z(U_2)_{\Gamma'}).$$
\end{lem}

Combining Proposition~\ref{prop-good-cover} and Corollary~\ref{cor-great}, we obtain our main formula for the intersection of special cycles. 

\begin{theo}\label{theo4.12} For totally positive subspaces $U_1$ and $U_2$ of $V$ with $\dim_F U_i=r_i$, 
$$[Z(U_1)_\Gamma]\cdot [Z(U_2)_\Gamma]  = \sum_{\gammabold} \cbold_{\Gamma}^{r_1+r_2-r(W_\gammabold)}\cap [Z(W_\gammabold)_{\Gamma}] \quad \in  \ch^{(r_1+r_2)d_+}(S_{\Gamma}),$$
where $\gammabold$ runs over the index set 
$\Gamma\back \big( \,\Gamma/\Gamma_{U_1}\times \Gamma/\Gamma_{U_2})$.
\end{theo}
\begin{proof}
For totally positive subspaces $U_1$ and $U_2$ of $V$, there is a subgroup $\Gamma'$ of finite index in $\Gamma$ such that 
$$\pr^*(Z(U_1)_\Gamma\cdot Z(U_2)_\Gamma)  = \sum_{\gammabold} \cbold_{\Gamma'}^{r_1+r_2-r(W_\gammabold)}\cap [Z(W_\gammabold)_{\Gamma'}] \quad \in  \ch^{(r_1+r_2)d_+}(S_{\Gamma'}),$$
where $\gammabold$ runs over (\ref{cone-comp-index}). 
This follows from the discussion following Proposition~6.1 in Fulton where the irreducible components of the normal cone are described by Lemma~\ref{lem-cone-comp}, 
and their contributions, according to Example~6.1.1 of Fulton, are given by Proposition~\ref{prop-good-cover} and Corollary~\ref{cor-great}. 

We will need the following fact. 
\begin{lem}\label{little-lem} Let $\Gamma_\gammabold$ be  stabilizer in $\Gamma$ of the coset $\gammabold \in \Gamma/\Gamma_{U_1} \times \Gamma/\Gamma_{U_2}$. 
Then $\Gamma_{\gammabold}= \Gamma_{W_\gammabold}$.
\begin{proof} Elements of $\Gamma_{\gammabold}$ preserve the subspace $W_\gammabold$ and hence lie in $\tilde \Gamma_{W_\gammabold}=\Gamma_{W_\gammabold}$, in the notation 
of (i) of Lemma~\ref{jaffee.lem}.  Conversely, if $\gamma_0\in \Gamma_{W_\gammabold}$, then $\gamma_0$ acts trivially on $W_\gammabold$ and hence it preserves the subspaces
$\gammabold_1 U_1$ and $\gammabold_2 U_2$. It then lies in both $\Gamma_{\gammabold_1 U_1}=\gammabold_1 \Gamma_{U_1}\gammabold_1^{-1}$ 
and  $\Gamma_{\gammabold_2 U_2}=\gammabold_2 \Gamma_{U_2}\gammabold_2^{-1}$ and hence in $\Gamma_\gammabold$. 
\end{proof} 
\end{lem} 
Now we pass back to our original $\Gamma$. 
For $\gammabold\in \Gamma/\Gamma_{U_1}\times \Gamma/\Gamma_{U_2}$, let 
$r(\gammabold) = \dim_F W_\gammabold$. 
With this notation and using Lemma~\ref{little-lem}, we have
\begin{align*}
\sum_{\substack{ \gammabold \\ \snass r(\gammabold)=r \\ \snass \mod \Gamma'}}  [Z(W_\gammabold)_{\Gamma'}] &= \sum_{\substack{ \gammabold \\ \snass r(\gammabold)=r \\ \snass \mod \Gamma}}  
\sum_{\gamma\in \Gamma'\back \Gamma/\Gamma_{W_\gammabold}} [Z(\gamma W_\gammabold)_{\Gamma'}] \\
\nass
{}&= \sum_{\substack{ \gammabold \\ \snass r(\gammabold)=r \\ \snass \mod \Gamma}}\pr^*([Z(W_\gammabold)_\Gamma]).
\end{align*}
Thus 
$$\pr^*(Z(U_1)_\Gamma\cdot Z(U_2)_\Gamma)  = \sum_r  \cbold_{\Gamma'}^{r_1+r_2-r}\cap 
\bigg(\  \sum_{\substack{ \gammabold \\ \snass r(\gammabold)=r \\ \snass \mod \Gamma}}\pr^*([Z(W_\gammabold)_\Gamma])\ \bigg),$$
and hence
\begin{align*}
\pr_*\pr^*(Z(U_1)_\Gamma\cdot Z(U_2)_\Gamma)  &= \sum_r  \pr_*\bigg(\ \cbold_{\Gamma'}^{r_1+r_2-r}\cap 
\bigg(\  \sum_{\substack{ \gammabold \\ \snass r(\gammabold)=r \\ \snass \mod \Gamma}} \pr^*([Z(W_\gammabold)_\Gamma])\ \bigg)\ \bigg)\\
\nass
{}&=\sum_r  \pr_*(\ \cbold_{\Gamma'}^{r_1+r_2-r}\ )\cap 
\bigg(\  \sum_{\substack{ \gammabold \\ \snass r(\gammabold)=r \\ \snass \mod \Gamma}} [Z(W_\gammabold)_\Gamma])\ \bigg).
\end{align*}
But by definition $\cbold_{\Gamma'} = \pr^*( \cbold_\Gamma)$, and so, canceling the index $|\Gamma:\Gamma'|$ from both sides, we have 
$$Z(U_1)_\Gamma\cdot Z(U_2)_\Gamma = \sum_{\substack{\gammabold\\ \snass \mod \Gamma}} \cbold_{\Gamma}^{r_1+r_2-r(\gammabold)}\cap [Z(W_\gammabold)_{\Gamma}]$$
as claimed.
\end{proof}

\section{Generating series for special cycles: ad\`elic version}\label{section5}

In this section, we formulate a version of the generating series for special cycles in ad\`elic language using cycles weighted by Schwartz functions 
on the finite ad\'eles of $V$. 
This is a slight variation\footnote{In previous cases the parameter $d_+=1$.} of the setup of \cite{K.duke} also used in \cite{YZZ}.  
It is far more convenient than the classical setup of Section~\ref{section2} for pullback arguments and intersection relations. 

Most of the definitions and results of \cite{K.duke}, where $d_+=1$,  remain unchanged when $d_+$ is arbitrary. We will continue to assume that 
$d_+<d$, however, so that our varieties will be proper over $\C$.  We will not repeat all of the discussion of \cite{K.duke}, but will simply recall the
main results and note where minor modifications are needed. There will be a slight shift in notation. 

\subsection{Weighted cycles}

Let $G=R_{F/\Q}\GSpin(V)$ and for a compact open subgroup $K \subset G(\A_f)$, let 
\beq\label{def-SK}
S_K = G(\Q)\back D\times G(\A_f)/K.
\eeq
If $K$ is neat, $S_K$ is a union of smooth projective varieties as discussed in Section~\ref{section2}. 
More precisely, we write 
$$G(\A_f) = \bigsqcup_j G(\Q)_+ g_j K,$$
where $G(\Q)_+$ is the subgroup of element with totally positive spinor norm, and, for $g\in G(\A_f)$, we let $\Gamma_g = G(\Q)_+\cap g K g^{-1}$. 
Then, we have the decomposition,  \cite{K.duke} (1.5),  
$$S_K \simeq \bigsqcup_j \Gamma_{g_j}\back \Dp = \bigsqcup_j S_{\Gamma_{g_j}}.$$

 If $U$ is a totally positive definite subspace of $V$,  $Z(U)_{\Gamma_g}$ is 
a connected cycle of codimension $r(U)d_+$  in $S_{\Gamma_g}$.   In the notation of \cite{K.duke}, (3.3), $Z(U)_{\Gamma_g}=c(U,g,K)$.  

The weighted cycles are defined as follows. 
For $\ph \in S(V(\A_f)^n)$ and $T \in \Sym_n(F)_{\ge0}$ totally positive semi-definite, 
let
\beq\label{def-wt-cycle}
Z(T,\ph,K) := \sum_j \sum_{\substack{x\in \O_T(F) \\ \snass \mod \Gamma_{g_j}}} \ph(g_j^{-1} x) \, [\,Z(U(x))_{\Gamma_{g_j}}\,] \cap \cbold_{\Gamma_{g_j}}^{n-r(x)},
\eeq
where 
\beq\label{def-O-T}
\O_T = \{\ x\in V^n \mid Q(x) = T\ \}.
\eeq
Then $Z(T,\ph,K)$ is an element of $\ch^{n d_+}(S_K)\tt_\Q R$, where $R$ is the subfield of $\C$ where $\ph$ takes values.  
We will take $R=\C$ from now on. 
Note that 
\beq\label{0-class}
Z(0,\ph,K) = \cbold_S^n\cdot \ph(0),
\eeq
where the restriction of $\cbold_S$ to $S_{\Gamma_{g_j}}$ is $\cbold_{\Gamma_{g_j}}$. 

\begin{rem} Here we have taken as our definition the analogue of the expression given in Proposition~5.4 of \cite{K.duke} in the case $d_+=1$.
The alternative definition in terms of `natural' cycles and the the proof of the coincidence of the two definitions is given in Section~\ref{section10}. 
\end{rem}

The equivariance and pullback properties of Propositions~5.9 and 5.10 of \cite{K.duke} go over without change\footnote{It should be noted however, that 
the proofs depend on the relations between the connected cycles and the `natural' cycles of Section 2 of \cite{K.duke} and Section\ref{section10} below. These relations depend, in turn, on 
Lemma~5.7 of \cite{K.duke}, which is due to Weil.}. \hfb 
For any $g\in G(\A_f)$, 
\beq\label{g-dot}
Z(T,\o(g)\ph,gKg^{-1}) = Z(T,\ph,K)\cdot g^{-1}.
\eeq
If $K$ is neat and $K'\subset K$ is another compact open subgroup of $G(\A_f)$, then
\beq\label{K-lim}
\pr^*(Z(T,\ph,K)) = Z(T,\ph, K'),
\eeq
where $\pr:S_{K'}\ra S_K$ is the natural projection. As a result, we have well defined classes 
\beq\label{dir-lim}
Z(T,\ph) \in \ch^{nd_+}(S) := \varinjlim_K \ch^{nd_+}(S_K).
\eeq

In the case $d_+=1$, the following result was suggested in \cite{K.duke}, Remark~6.3, and proved in \cite{YZZ}, Proposition~2.6.
\begin{prop}\label{prop5.1}  For $T_i\in \Sym_{n_i}(F)_{\ge0}$ and $\ph_i\in S(V(\A_f)^{n_i})^K$, 
$$Z(T_1,\ph_1,K)\cdot Z(T_2,\ph_2,K) = \sum_{\substack{ T \in \Sym_{n_1+n_2}(F)_{\ge0}\\ \snass T = \bpm \scr T_1&*\\ \scr {}^t*&\scr T_2\epm} } Z(T,\ph_1\tt\ph_2,K)\ \in \ch^{(n_1+n_2)d_+}(S_K).$$
\end{prop} 
\begin{rem} By invariance under pullback, for classes in the direct limit (\ref{dir-lim}),  we have
\beq\label{dir-lim-prod}
Z(T_1,\ph_1)\cdot Z(T_2,\ph_2) = \sum_{\substack{ T \in \Sym_{n_1+n_2}(F)_{\ge0}\\ \snass T = \bpm \scr T_1&*\\ \scr {}^t*&\scr T_2\epm} } Z(T,\ph_1\tt\ph_2)\ 
\in \ch^{(n_1+n_2)d_+}(S).
\eeq
\end{rem}
\begin{proof} Choose $K$ neat such that $\ph_1$ and $\ph_2$ are $K$-invariant and compute
\begin{align*}
Z(T_1,&\ph_1,K)\cdot Z(T_2,\ph_2,K)\\
\nass
{}&= \sum_j \sum_{\substack{x_1\in \O_{T_1}(F) \\ \snass \mod \Gamma_{g_j}}}\sum_{\substack{x_2\in \O_{T_2}(F) \\ \snass \mod \Gamma_{g_j}}}
\ph_1(g_j^{-1} x_1)\,\ph_2(g_j^{-1} x_2)\\
\nass
{}&\hskip 1.5in \times [\,Z(U(x_1))_{\Gamma_{g_j}}\,]\cdot [\,Z(U(x_2))_{\Gamma_{g_j}}\,] \cap \cbold_{\Gamma_{g_j}}^{n_1+n_2-r(x_1)-r(x_2)}.
\end{align*}
By Theorem~\ref{theo4.12}, the intersection number is given by 
$$[\,Z(U(x_1))_{\Gamma_{g_j}}\,]\cdot [\,Z(U(x_2))_{\Gamma_{g_j}}\,] = \sum_r \cbold_{\Gamma}^{r_1+r_2-r}\cap \bigg(\ \sum_\gammabold [Z(W_\gammabold)_{\Gamma_{g_j}}]\ \bigg)$$
where $\gammabold$ runs over the $\Gamma_{g_j}$-orbits in the set
$$
\{\,\gammabold\in \Gamma_{g_j}/\Gamma_{g_j,U(x_1)} \times \Gamma_{g_j}/\Gamma_{g_j, U(x_2)} \mid r(\gammabold)=r\, \}.
$$
The whole intersection number then unwinds to 
\begin{align*}
Z(T_1,&\ph_1,K)\cdot Z(T_2,\ph_2,K)\\
\nass
{}&= \sum_j \sum_{\substack{ T \in \Sym_{n_1+n_2}(F)_{\ge0}\\ \snass T = \bpm \scr T_1&*\\ \scr {}^t*&\scr T_2\epm} } \sum_{\substack{x\in \O_{T}(F) \\ \snass \mod \Gamma_{g_j}}}
(\ph_1\tt\ph_2)(g_j^{-1} x)\, [\,Z(U(x))_{\Gamma_{g_j}}\,] \cap \cbold_{\Gamma_{g_j}}^{n_1+n_2-r(x)}, 
\end{align*}
and this gives the claimed expression. 
\end{proof}

\subsection{The generating series}
Here we use the notation of Sections 7 and 8 of \cite{K.duke}.

For $g\in G(\A_f)$, $\tau\in \H_n^d$ and $\ph\in S(V(\A_f)^n)^K$, define the formal generating series
\beq\label{adelic-gen}
\phi_n(\tau;\ph,K) = \sum_{T\in \Sym_n(F)_{\ge0}} [Z(T,\ph,K)]\, \qq^T\ \in \ch^{nd_+}(S_{K})_\C[[\qq]].
\eeq
Note that for $g\in G(\A_f)$, by (\ref{g-dot}),  we have 
$$\phi_n(\tau;\ph,K)\cdot g^{-1} = \sum_{T\in \Sym_n(F)_{\ge0}} [Z(T,\o(g)\ph,gKg^{-1})]\, \qq^T\ \in \ch^{nd_+}(S_{g Kg^{-1}})_\C[[\qq]].$$
Passing to the limit over $K$, noting (\ref{K-lim}), and writing 
$$\phi_n(\tau;\ph) = \sum_{T\in \Sym_n(F)_{\ge0}} [Z(T,\ph)]\, \qq^T\ \in \ch^{nd_+}(S)_\C[[\qq]],$$
we see that the formal series is equivariant with respect to the 
actions of $G(\A_f)$ on $S(V(\A_f)^n)$ and on $\ch^{nd_+}(S)$.

As a consequence of the product formula (\ref{dir-lim-prod}), we have a product formula for the formal generating series. 
For $n=n_1+n_2$ with $n_i\ge 1$, let 
$$J: \H_{n_1}^d\times \H_{n_2}^d \lra \H_n^d, \qquad (\tau_1,\tau_2) \mapsto \bpm \tau_1&{}\\{}&\tau_2\epm.$$
\begin{prop} For $n=n_1+n_2$ with $n_i\ge1$ and for $\ph_1\in S(V(\A_f)^{n_1})$ and $\ph_2\in S(V(\A_f)^{n_2})$, 
\beq\label{gen-fun-prod}
\phi_n(\bpm \tau_1&{}\\{}&\tau_2\epm;\ph_1\tt\ph_2) = \phi_{n_1}(\tau_1;\ph_1)\cdot \phi_{n_2}(\tau_2;\ph_2).
\eeq
In particular, the $T_2$-coefficient with respect to $\qq_2$ of the pullback is given by 
\beq\label{gen-fun-prod-alt}
\phi_n(\bpm \tau_1&{}\\{}&\tau_2\epm;\ph_1\tt\ph_2)_{T_2} = \phi_{n_1}(\tau_1;\ph_1)\cdot Z(T_2,\ph_2).
\eeq
and the $\qq_2$-constant term of the pullback
\beq\label{gen-fun-prod-alt-0}
\phi_n(\bpm \tau_1&{}\\{}&\tau_2\epm;\ph_1\tt\ph_2)_{0} =\ph_2(0) \sum_{T_1} Z(T_1,\ph_1)\cdot \cbold_S^{n-n_1}\ \qq_1^{T_1}
\eeq

\end{prop} 
\begin{proof}
\begin{align*}
 \phi_{n_1}(\tau_1;\ph_1)\cdot \phi_{n_2}(\tau_2;\ph_2)&= \sum_{T_1,T_2} Z(T_1,\ph_1)\cdot Z(T_2,\ph_2)\,\qq_1^{T_1}\,\qq_2^{T_2}\\
 \nass
 {}&= \sum_{T_1,T_2} \sum_{ T = \bpm \scr T_1&*\\ \scr {}^t*&\scr T_2\epm}  Z(T,\ph_1\tt\ph_2)\  J^* \qq^T.
\end{align*}
\end{proof} 
\begin{rem}  In \cite{K.duke}, such a product formula for generating series valued in cohomology groups was a consequence of the 
product formula for the theta forms (\ref{theta-form}), and was used to prove the cup product version of (\ref{dir-lim-prod}). Here we have given a 
direct geometric proof of (\ref{dir-lim-prod}) and obtain (\ref{gen-fun-prod}) from it. 
\end{rem} 

Using the definition (\ref{def-wt-cycle}) of the weighted cycles, we can also write our formal series as 
\beq\label{theta-form}
\phi_n(\tau;\ph,K) = \sum_j \sum_{\substack{x\in V(F)^n \\ \snass \mod \Gamma_{g_j}}} \ph(g_j^{-1} x) \, [\,Z(U(x))_{\Gamma_{g_j}}\,] \cap \cbold_{\Gamma_{g_j}}^{n-r(x)}\cdot  \qq^{Q(x)},
\eeq
a kind of formal theta-function.

For any complex valued linear functional $\l$ on $\ch^{nd_+}(S_K)$, 
let
\beq\label{power-series}
\phi_n(\tau;\ph,K,\l) = \sum_{T\in \Sym_n(F)_{\ge0}} \l\big(\ [Z(T,\ph,K)]\ \big)\, \qq^T \in \C[[\qq]].
\eeq

Recall that the generating series $\phi_n(\tau;\ph,K)$ is said to be modular if, for every linear functional $\l$, the 
formal power series (\ref{power-series}) is absolutely convergent and the resulting holomorphic function on $\H_n^{d}$ 
is a Hilbert-Siegel modular form.

\subsection{Classes in cohomology} 
Taking $\l$ to be the cycle class map $\text{\rm cl}: \ch^{n d_+}(S_K) \ra H^{2n d_+}(S_K)$, we have 
\beq\label{KM-modular}
\phi_n(\tau;\ph,K,\text{\rm cl}) =\sum_{T\in \Sym_n(F)_{\ge0}}\text{\rm cl}( [Z(T,\ph,K)]\,)\, \qq^T\  \in H^{2 n d_+}(S_K).
\eeq
The modularity of this series is the analogue with Theorem~8.1 of \cite{K.duke} for $d_+$ arbitrary. 
We briefly sketch the argument, referring to Section 7 of \cite{K.duke} for details.

Let $G' = R_{F/\Q} \Sp(n)$ and let $\wtg(\A)$ be the metaplectic cover of $G'(\A)$.
 The metaplectic group $\wtg(\A)$ acts on $S(V(\A)^n)$ via the Weil representation $\o=\o_\psi$ determined by the additive character $\psi$ of 
$F_\A/F$. This action commutes with the linear action of $g\in G(\A)$, $\o(g)\ph(x) = \ph(g^{-1} x)$. 

Let 
\beq\label{arch-SF}
\phbold_\infty = \underset{d_+}{\underbrace{\ph^{(n)}\tt \dots \tt \ph^{(n)}}} \tt \underset{d-d_+}{\underbrace{\ph_+^0 \tt \dots \tt \ph_+^0}} \ \in [S(V_\R^n) \tt A^{(nd_+,nd_+)}(D)]^{G(\R)},
\eeq
where, for $1\le j \le D_+$, $\ph^{(n)}\in S(V_j^n)\tt A^{(n,n)}(D_j)$ is the Schwartz form for $V_j$, and, for $j>d_+$, $\ph_+^0\in S(V_j^n)$ is the Gaussian for $V_j$, cf. Sections 7 and 8 of \cite{K.duke}.

For use later, we define a $(1,1)$-form on $D^{(j)}$ by 
\beq\label{def-Omega}
\Om_j = \ph^{(1)}(0),
\eeq
and note that $\ph^{(n)}(0) = \Om_j^n$. 
By Corollary~4.12 of \cite{K.Bints}, the form $\Om_j$ on the factor $D^{(j)}$  is the first Chern form of the inverse of the tautological bundle on the 
space of oriented negative $2$-planes in 
$V_j$. 
Thus 
\beq\label{Om-top}
\Om_S^n = \phbold_\infty(0)
\eeq
is an $(nd_+,nd_+)$-form representing the class $\cbold_S^n$, the $n$-th power of the top Chern class of the vector bundle 
$\boldCC_S$ defined in  (\ref{bold-CS}). 

For $\ph \in S(V(\A_f)^n)^K$, let $\widetilde{\ph} = \phbold_\infty\tt \ph \in S(V(\A)^n)$. Then the theta series 
\beq\label{theta-form}
\theta(g',g;\ph) = \sum_{x\in V(F)^n} \o(g')\widetilde{\ph}(g^{-1}x), \qquad g\in G(\A_f), \ g'\in \wtg(\A),
\eeq
defines a closed $(nd_+,nd_+)$-form on $S_K$ and is left invariant for the (canonical) image of $G'(\Q)$ in $\wtg(\A)$. 
When $g=1$, we will write simply $\theta(g';\ph)$ for this series.  

Let $Mp(n,\R)$ be the metaplectic cover of $\Sp(n,\R)$, and, for $g_0'\in Mp(n,\R)$, write 
$$g_0' = \bpm 1&u\\{}&1\epm \bpm v^{\frac12} &{}\\{}&{}^tv^{-\frac12}\epm k', \qquad \tau = u+i v\in \H_n,$$
as in (7.21) of \cite{K.duke}, where $k'\in K'$, the $2$-fold cover of $U(n)$.  Then for $T\in \Sym_n(\R)$, define the Whittaker function 
$$W_T(g_0') = \det(v)^{\frac{m+2}4} e(\tr(T\tau))\,\det(k')^{\frac{m+2}4}.$$
When $k'=1$, we write $g'_0= g'_\tau$ and note that 
$$\det(v)^{-\frac{m+2}4}\,W_T(g'_\tau) = e(\tr(T\tau)).$$
There is a surjection $Mp(n,\R)^d \ra \wtg(\R)$ and an inclusion 
$\wtg(\R) \hookrightarrow \wtg(\A)$.  For $T\in \Sym_n(F)$ and $g'\in \wtg(\R)$, let 
$$W_T(g') = W_{T_1}(g'_1)\dots W_{T_d}(g'_d),$$
where $T_j = \s_j(T)$ and $(g_1,\dots, g_d)\in Mp(n,\R)^d$ is a preimage of $g'$.
Note that 
$$\qq^T=N(\det(v))^{-\frac{m+2}4}\,W_T(g'_\tau),$$
where $N(\det(v)) = \prod_j \det(v_j)$. 

The analogue of Theorem~8.1 of \cite{K.duke} is that, as a consequence of \cite{KM1}, \cite{KM2} and \cite{KM3}, 
the cohomology class\footnote{Which we denote by  $[[\,\theta(g';\ph)\,]]$.}
of $\theta(g';\ph)$ is given by 
\beq\label{KM-relation}
[[\,\theta(g'_\tau;\ph)\,]] = N(\det(v))^{\frac{m+2}4} \sum_{T\in \Sym_n(F)_{\ge 0} }\text{\rm cl}([Z(T,\ph)])\,\qq^T \ \in H^{2nd_+}(S).
\eeq
The invariance of $\theta(g';\ph)$ under $G'(\Q)$ implies that the series (\ref{KM-modular})  is the $\qq$-expansion
of a Hilbert-Siegel modular form of weight $(\frac{m}2+1, \dots, \frac{m}2+1)$.

\section{A pullback formula}\label{section6}

\subsection{The classical version}
In this section we suppose that $U_0$ is a totally positive subspace of $V$ and let  
$$\rho: \Dp_0 = \Dp_{U_0} \lra \Dp.$$
Let $\Gamma_0=\Gamma_{U_0}$ and
suppose that $\s_0\Gamma\s_0 = \Gamma$, where $\s_0=\s_{U_0}$,  so that we have a regular embedding 
$$\rho:S^0_{\Gamma} = Z(U_0)_\Gamma = \Gamma_0\back \Dp_0  \lra \Gamma\back \Dp =S_\Gamma$$
of non-singular varieties. There is then a pullback homomorphism 
$$\rho^*:\ch^\bullet(S_\Gamma) \lra \ch^\bullet(S^0_\Gamma)$$
of Chow rings. Note that the tautological bundles $\mathcal L_j$ on $S_\Gamma$ pull back to the corresponding tautological bundles on $S^0_\Gamma$ and hence
 $\rho^* \cbold_S = \cbold_{S^0}$. 
The proof of the following result is analogous to that of Theorem~\ref{theo4.12} and will be omitted. 
\begin{prop}\label{pullback-classical}  Suppose that $U_0$ is a totally positive definite subspace of $V$, as above. Then for any totally positive subspace $U$ in $V$, 
$$\rho^*([Z(U)]_\Gamma) = \sum_{\gamma\in \Gamma_0\back \Gamma/\Gamma_U} [Z(\pr_0(\gamma U))_{\Gamma_0}]\cap \cbold_{S^0}^{(r(U)-r(\pr_0(\gamma U)))},$$
where $\pr_0:V \ra U_0^\perp$ is the orthogonal projection. 
\end{prop}

\subsection{The ad\`elic version} For a totally positive subspace $U_0$ in $V$, let $G^0 = R_{F/\Q}\text{\rm GSpin}(U_0^\perp)$ s
o that $G^0\subset G = R_{F/\Q}\text{\rm GSpin}(V)$. 
For a compact open subgroup $K\subset G(\A_f)$ and $g\in G(\A_f)$, let $K^0_g = G^0(\A_f)\cap g Kg^{-1}$ and let 
$$\rho_{g,K}: S^0_{K^0_g} = G^0(\Q)\back D_0\times G^0(\A_f)/K^0_g\lra S_K = G(\Q)\back D\times G(\A_f)/K, \qquad [z,h] \mapsto [z,hg],$$
be the natural morphism. Suppose that $K$ is neat. The corresponding pullback homomorphisms
$$\rho^*_{g,K}: \ch^\bullet(S_K) \lra \ch^\bullet(S^0_{K^0_g})$$ 
are compatible with the systems of projections (\ref{dir-lim}) and define a homomorphism 
$$\rho^*_g: \ch^\bullet(S) \lra \ch^\bullet(S^0)$$
on the direct limits.

\begin{prop}\label{prop-6.2}  For a Schwartz function $\ph\in S(V(\A_f)^n)^K$ and $T\in \Sym_n(F)$ totally positive semi-definite, 
$$\rho^*_{g,K}  (Z(T,\ph,K)) =  \sum_r \sum_{\substack{x_0\in U_0(F)^n}}  \ph^0_r(x_0)\, Z(T-Q(x_0),\ph_r^1,K^0_g).$$
where
\beq\label{ph-decomp}
\o(g) \ph = \sum_r \ph_r^0 \tt \ph_r^{1} \in S(U_0(\A_f))\tt S(U_0^\perp(\A_f)).
\eeq
\end{prop} 

\begin{proof}  The description of $\rho_{g,K}$ on connected components is given in Section 4 of \cite{K.duke}. 
Write
$$G(\A_f) = \coprod_j G(\Q)_+ g_j K,$$
and 
$$G^0(\A_f) = \coprod_i G^0(\Q)_+ h_i K^0_g,$$
and, for each $i$, write
$$G(\Q)_+ h_i g K = G(\Q)_+ g_j K,\qquad j = j(i), \quad h_i g = \gamma_i^{-1} g_j k_i.$$
Here the index $j=j(i)$ depends on $i$. 
Let 
\beq\label{match-ind}
\Gamma^0_i = G^0(\Q)_+\cap h_i K^0_g h_i^{-1} = G^0(\Q)_+ \cap h_i g K g^{-1} h_i^{-1} = G^0(\Q)_+\cap \gamma_i^{-1}g_j K g_j^{-1} \gamma_i,
\eeq
and $\Gamma_j = \Gamma_{g_j}=G(\Q)_+\cap g_j K g_j^{-1}$. Let 
$$\pi^0_i: D^{0,+} \lra \Gamma^0_i\back D^{0,+}\qquad\text{and}\qquad \pi_j: \Dp\lra \Gamma_j\back \Dp$$
be the projections.  Here note that $\Gamma^0_i$ is a subgroup of $\gamma_i^{-1} \Gamma_j \gamma_i$. Then
$$\rho_{g,K}:\ \coprod_i \Gamma^0_i\back D^{0,+} \simeq S^0_{K^0_g} \ \lra \ S_K \simeq 
\coprod_j \Gamma_j \back \Dp,\qquad \rho_i: \pi^0_i(z)\  \mapsto \ \pi_j(\gamma_i z),$$
where $j=j(i)$ as in (\ref{match-ind}) and $\rho_i: \Gamma_i^0\back D^{0,+}\rightarrow \Gamma_j\back D^+$ is the restriction of $\rho_{g,K}$ 
to the component $\Gamma_i^0\back D^{0,+}$. 
For $\ph\in S(V(\A_f)^n)^K$, the part of the special cycle $Z(T,\ph, K)$ lying in the connected component $S_{\Gamma_j}=\Gamma_j\back \Dp$ is given by 
\beq\label{j-part}
\sum_{\substack{x\in \O_T(F) \\ \snass \mod \Gamma_{j}}} \ph(g_j^{-1} x) \, [\,Z(U(x))_{\Gamma_{j}}\,] \cap \cbold_{\Gamma_{g_j}}^{n-r(x)}.
\eeq
For a fixed $i$ and taking $j=j(i)$, write $\rho_i = [\gamma_i] \circ \rho_i^\nat$ where $[\gamma_i]: S_{\gamma_i^{-1} \Gamma_j \gamma_i} 
\isoarrow S_{\Gamma_j}$ is the isomorphism 
associated to $\gamma_i\in G(\Q)_+$, as in (\ref{brak-eta}).   For convenience, we write $\Gamma_j' = \gamma_i^{-1} \Gamma_j \gamma_i$. 
By Proposition~\ref{pullback-classical}, the pullback of (\ref{j-part}) under $\rho_i$ is 
\begin{align*}
&\sum_{\substack{x\in \O_T(F) \\ \snass \mod \Gamma_{j}}} \ph(g_j^{-1} x) \, \rho_i^*\big(\ [\,Z(U(x))_{\Gamma_{j}}\,] \cap \cbold_{\Gamma_{j}}^{n-r(x)}\ \big)\\
\nass
{}&=\sum_{\substack{x\in \O_T(F) \\ \snass \mod \Gamma_{j}}} \ph(g_j^{-1} x) \, (\rho_i^\nat)^*\big(\ [\,Z(\gamma_i^{-1} U(x))_{\Gamma'_{j}}\,] \cap 
\cbold_{\Gamma'_{j}}^{n-r(x)}\ \big)\\
\nass
{}&=\sum_{\substack{x\in \O_T(F) \\ \snass \mod \Gamma'_{j}}} \ph(g_j^{-1} \gamma_i x) \, (\rho_i^\nat)^*\big(\ [\,Z(U(x))_{\Gamma'_{j}}\,] \cap 
\cbold_{\Gamma'_{j}}^{n-r(x)}\ \big)\\
\nass
{}&=\sum_{\substack{x\in \O_T(F) \\ \snass \mod \Gamma'_{j}}} \ph(g_j^{-1}\gamma_i x) \, \sum_{\gamma\in \Gamma^0_i\back \Gamma'_j/\Gamma'_{j,U(x)}} 
[\,Z(\pr_0(\gamma U(x)))_{\Gamma^0_{i}}\,] \cap \cbold_{\Gamma^0_i}^{r(U(x))-r(\pr_0(\gamma U(x)))}\cap \cbold_{\Gamma^0_{i}}^{n-r(x)}\\
\nass
{}&=\sum_{\substack{x\in \O_T(F) \\ \snass \mod \Gamma^0_{i}}} \ph(g_j^{-1}\gamma_i x) \,
[\,Z(\pr_0(U(x)))_{\Gamma^0_{i}}\,] \cap \cbold_{\Gamma^0_i}^{n-r(\pr_0(U(x)))}\\
\nass
{}&=\sum_{\substack{x_0\in U_0(F)^n,\ x_1\in U_0^\perp(F)^n\\ T = Q(x_0)+Q(x_1) \\ \snass \mod \Gamma^0_{i}}} \ph(g^{-1} (x_0+h_i^{-1}x_1))\, 
[\,Z(U(x_1))_{\Gamma^0_{i}}\,] \cap \cbold_{\Gamma^0_i}^{n-r(x_1)}.
\end{align*}
Here in the last line we note that $g_j \gamma_i = k_i g^{-1} h_i^{-1}$ and that $h_i$ acts trivially on $U_0$. Using (\ref{ph-decomp}), 
the sum on $i$ of the last expression is 
\begin{align*}
&\sum_i \sum_r\sum_{\substack{x_0\in U_0(F)^n,\ x_1\in U_0^\perp(F)^n\\ T = Q(x_0)+Q(x_1) \\ \snass \mod \Gamma^0_{i}}} \ph^0_r(x_0)\ph^1_r(h_i^{-1}x_1))\,
[\,Z(U(x_1))_{\Gamma^0_{i}}\,] \cap \cbold_{\Gamma^0_i}^{n-r(x_1)}\\
\nass
{}&=\sum_r\sum_{\substack{x_0\in U_0(F)^n}}  \ph^0_r(x_0) \sum_i
\sum_{\substack{x_1\in U_0^\perp(F)^n\\ Q(x_1)= T-Q(x_0) \\ \snass \mod \Gamma^0_{i}}} 
\ph^1_r(h_i^{-1}x_1))\,
[\,Z(U(x_1))_{\Gamma^0_{i}}\,] \cap \cbold_{\Gamma^0_i}^{n-r(x_1)}\\
\nass
{}&= \sum_r \sum_{\substack{x_0\in U_0(F)^n}}  \ph^0_r(x_0)\,[\, Z(T-Q(x_0),\ph_r^1,K^0_g)\,].
\end{align*}
\end{proof}

\subsection{The pullback of the generating series}
Applying Proposition~6.2, we obtain a formula for the pullback of the generating series. 
\begin{prop}\label{main-pullback}  With the notation of the previous section, suppose that $\ph\in S(V(\A_f)^n)^K$ satisfies (\ref{ph-decomp}). 
Then 
$$\rho_{g,K}^*\big(\ \phi_n(\tau;\ph,K)\ \big) = \sum_r \theta(\tau,\ph^0_r)\cdot \phi_n(\tau, \ph^1_r,K^0_g)$$
where 
$$\theta(\tau,\ph^0_r) = \sum_{x_0\in U_0(F)^n} \ph^0_r(x_0) \, \qq^{Q(x_0)}$$
and 
$$\ph_n(\tau, \ph^1_r,K^0_g) \ \in \ch^{nd_+}(S^0_{K^0_g})[[\qq]]$$ 
is the formal generating series for $S^0_{K^0_g}$. 
\end{prop}

Note that the decomposition (\ref{ph-decomp}) depends on $g$. 

\section{The  embedding trick}\label{section8}
We now slightly vary the situation of Section~\ref{section6}. Let $U_0$ be a totally positive definite space over $F$ of dimension $4\ell$
and let $\widetilde{V} = U_0\oplus V$ be the orthogonal sum. The signature of $\widetilde{V}$ is given by (\ref{sig-embed}). 
Let $\widetilde{G} = R_{F/\Q} \text{GSpin}(\widetilde{V})$ so that there is a natural homomorphism $G\ra \widetilde{G}$.  
For $\widetilde{K}$ compact open in $\widetilde{G}(\A_f)$ and $K = G(\A_f)\cap \widetilde{K}$, we obtain a morphism 
$$\rho_{\widetilde{K}}: S_K \lra \widetilde{S}_{\widetilde{K}},$$
and, assuming that $\widetilde{K}$ is neat so that these are smooth varieties,  a ring homomorphism 
$$\rho_{\widetilde{K}}^*: \ch^\bullet(\widetilde{S}_{\widetilde{K}}) \lra \ch^\bullet(S_K).$$
Passing to the limit over $\widetilde{K}$, we also have 
$$\rho^*: \ch^\bullet(\widetilde{S}) \lra \ch^\bullet(S).$$

For $\ph\in S(V(\A_f)^n)$ and $\ph^0\in S(U_0(\A_f)^n)$, Proposition~\ref{main-pullback} yields 
\beq\label{embed-step1} 
\rho^*\big(\ \phi_n(\tau;\ph^0\tt \ph)\ \big) =  \theta(\tau,\ph^0)\cdot \phi_n(\tau, \ph).
\eeq
Thus the two formal generating series are related by multiplication by a holomorphic Hilbert-Siegel theta series. 
The following result will be proved in the next section. 

\begin{prop}\label{prop7.1}  Suppose that for all choices of $\ph^0\in S(U_0(\A_f)^n)$ the series $\phi_n(\tau;\ph^0\tt \ph)$ is modular. 
Then the series $\phi_n(\tau, \ph)$ is modular. 
\end{prop}  

As explained in \cite{YZZ}, we have the following non-vanishing result. 
\begin{lem} For any point $\tau_0\in \H_n^d$, there exists a function $\ph^0\in S(U^0(\A_f)^n)$ such that 
$\theta(\tau_0,\ph^0)\ne0$. 
\end{lem} 
\begin{proof} 
Suppose that the linear functional 
$$\ph^0\mapsto \theta(\tau_0,\ph^0) = N(\det v_0)^{-\ell}\,\theta(g'_{\tau_0}, \ph^0), \qquad N(\det v_0) = \prod_{j=1}^d \det(v_{0,j}),$$ 
on $S(U^0(\A_f)^n)$ is zero. Then, for all $\ph^0$,  the function $\theta(g',\ph^0)$ on $G'(\A)$ vanishes on the subset $G'(\Q)g'_{\tau_0} G'(\A_f)$. 
But this set is dense in $G'(\A)$ and the functions $\theta(g',\ph^0)$ are not all zero.  
\end{proof}

\section{Formal Fourier series}\label{sectionFFS}
 
In this section, we prove Proposition~\ref{prop7.1}. 
The argument, using formal Fourier series and a special case of a result of \cite{knoeller},  was provided by Jan Bruinier.

Let 
$\mathcal S_F = \Sym_n(O_F)$ and let 
\begin{align*}
\mathcal S_F^\vee &= \{ x\in \Sym_n(F)\mid \tr_{F/\Q}\tr(xy) \in \Z, \ \forall y\in \mathcal S_F\ \}=  \Sym_n(\Z)^\vee \tt_\Z\d_F^{-1}.
\end{align*}
Let $\mathcal C$ be the cone of totally positive definite elements in $\Sym_n(\R)^d$ so that 
$$\mathcal S_F^\vee \cap \bar{\mathcal C} =\{ \ x\in \mathcal S_F^\vee\mid \s_j(x) \ge 0, \text{for all $j$}\}.$$

For a subgroup $\Gamma\subset \Sp_n(F)$ commensurable with $\Sp_n(O_F)$, there is a positive integer $\nu\in \Z_{>0}$ such that 
$\Gamma$ contains the principal congruence subgroup $\Gamma(\nu)$ of $\Sp_n(O_F)$.  It will be sufficient to consider the case
$\Gamma= \Gamma(\nu)$. We will assume that $\nu\ge 3$ to eliminate sign issues. 

Let $N= N_P$ (resp. $M=M_P$) be the unipotent 
radical  (resp. the standard Levi factor) of the Siegel parabolic $P$ of $\Sp_n/F$ and let 
\beq\label{unipotents}
\Gamma_N  = \Gamma\cap N(F)=\{ \ n(\b) \mid \b \in \nu\cdot \mathcal S_F\}, \qquad  n(\b) =\bpm 1_n&\b\\{}&1_n\epm,
\eeq
and 
$$\Gamma_M=\Gamma\cap M(F) = \{ \ m(\e)\mid \e\in \GL_n(O_F), \e\equiv 1_n\!\! \mod \nu O_F\ \}, \qquad m(\e) = \bpm \e&{}\\{}&{}^t\e^{-1}\epm.$$
We write 
$\L=\{\ \e\in \GL_n(O_F), \e\equiv 1_n\!\! \mod \nu O_F\ \}$. 

For $k\in \Z_{\ge0}$, let $M_k(\Gamma)$ be the vector space of Hilbert-Siegel modular forms of parallel 
weight\footnote{The case of half-integral weight can be formulated in exactly the same way using the metaplectic group. We leave this to the reader. }
 $k$ 
with respect to $\Gamma$. The graded ring $M_*(\Gamma) = \oplus_{k\ge 0} M_k(\Gamma)$ is an integral domain. We write $Q(M_*(\Gamma))$ for 
its quotient field. It can be viewed as a subfield of the field of meromorphic functions on $\Gamma_N\back \H_n^d$. 
A function $f\in M_k(\Gamma)$ has a Fourier expansion of the form
\beq\label{F-exp-f}
f(\tau) = a_f(0)+\sum_{T\in \sfs } a_f(T) \, \qq^T,
\eeq
where $\qq^T$ is given by (\ref{def-qT}) and where, to lighten notation, we write 
\beq\label{L-dot}
\sfs =( \nu^{-1} \cdot \mathcal S_F^\vee \setminus \{0\}) \cap \bar{\mathcal C}.
\eeq
This set, which  is denoted by $L^{\bigcdot}$ in \cite{knoeller}, depends on $\nu$, although we we omit this dependence from the notation. 
Note that the Fourier series is `symmetric' with respect to $\L$, i.e., for all $\e\in \L$, 
\beq\label{def-symmetric}
a_f(\e\cdot T) = a_f(T), \quad \e\cdot T = \e\,T\,{}^t\e.
\eeq

A formal Fourier series over $F$ of genus $n$ is a function $a: \Sym_n(F) \ra \C$. It can be viewed as a formal Laurent series 
$$\sum_{T\in \Sym_n(F)} a(T) \, \qq^T.$$
Define the support of such a series as 
$$\text{\rm supp}(a) = \{ \ T\in \Sym_n(F)\mid a(T)\ne 0\ \}.$$

Let $\FFS$ be the complex vector space of all such formal series and let 
$$\FFSS =\{ \ a\in \FFS\mid \text{supp}(a) \subset \sfs\cup\{0\}\ \}. $$
Note that $\FFSS$ is a ring for the product defined by 
$$a\cdot b = c, \qquad c(T) = \sum_{\substack{ R\in \sfsnull \\ \snass T-R \in \sfsnull}} a(R) \, b(T-R).$$
The sum is finite since, for $v\in \Sym_n(\R)_{\ge0}$ the set $\{w\in \Sym_n(\R)\mid w\ge 0, \ v-w \ge0\}$ is compact 
and the image of $\sfsnull$ is discrete in $\Sym_n(\R)^d$. 

An element $a\in \FFS$ is symmetric with respect to $\L$ if it satisfies (\ref{def-symmetric}). We denote by 
$\FFSS_\L$ the subring of symmetric elements in $\FFSS$.  We postpone the proof of the following crucial fact. 
 
\begin{prop}\label{prop8.1} $\FFSS_\L$ is an integral domain.
\end{prop}

There is a natural injective ring homomorphism
\[
\ph:M_*(\Gamma)\to \FFSS_\L, \quad f\longmapsto a_f(0)+\sum_{T\in \sfs } a_f(T) \cdot \qq^T,
\]
taking a holomorphic modular form to its Fourier series.
It induces an inclusion of quotient fields 
\[
Q(\varphi): Q(M_*(\Gamma))\to Q(\FFSS_\L),
\]
such that the diagram 
\begin{align}
\label{cd}
\xymatrix{M_*(\Gamma) \ar[d]\ar^\varphi[r]& \FFSS_\L \ar[d]\\ Q(M_*(\Gamma))\ar^{Q(\varphi)}[r]&Q(\FFSS_\L)}
\end{align}
commutes.

\begin{prop}\label{prop8.2}
Let $c\in \FFSS_\L$. Suppose the following conditions are satisfied.  
\begin{enumerate}
\item[(i)] There are modular forms $f\in M_{k+l}(\Gamma)$ and $g\in M_{l}(\Gamma)$ such that 
$$\varphi(f)= \varphi(g)\cdot c\ \in \FFSS_\L.$$ 
\item[(ii)] For any $z\in \H_n^d$ there exist holomorphic modular forms $f_{z}\in M_{k+l'}(\Gamma')$ and $g_{z}\in M_{l'}(\Gamma')$, 
where the weight $l'$ and $\Gamma'$, a congruence subgroup of $\Gamma$, may depend on $z$, such that 
\begin{enumerate}
\item[(a)] $\varphi(f_z)= \varphi(g_z)\cdot c\in \FFSS_{\L'}$,
\item[(b)] $g_z(z)\neq 0$.
\end{enumerate}
\end{enumerate}
Then there exists an $h\in M_k(\Gamma)$ such that $c=\varphi(h)$, that is, $c$ is the Fourier expansion of the holomorphic Hilbert-Siegel modular form $h$.
In particular, the series $c$ is absolutely convergent on $\H_n^d$. 
\end{prop}

\begin{proof}
By (i) and the diagram \eqref{cd}, we have 
\[
Q(\varphi)\left(\frac{f}{g}\right)=\frac{\varphi(f)}{\varphi(g)}= c \in Q(\FFSS_\L).
\]
Let $h= g^{-1}f$ so that $h$ is a meromorphic modular form for $\Gamma$ of weight $k$.
Let $z\in \H_n^d$ and let $\Gamma'$ be as in (ii)(a).  The inclusion $M_r(\Gamma) \ra M_r(\Gamma')$ induces an inclusion of the 
diagram (\ref{cd}) for $\Gamma$ into the corresponding diagram for $\Gamma'$, and we have
\[
Q(\varphi)\left(\frac{f_{z}}{g_{z}}\right)= c .
\]
Since $Q(\varphi)$ is injective, we obtain 
\[
h=\frac{f}{g}=\frac{f_{z}}{g_{z}}.
\]
Thus, by (ii)(b), $h$  is holomorphic in neighborhood of $z$ and hence is holomorphic on all of $\H^d_n$.
But this implies that $h\in M_k(\Gamma)$ and $c=\varphi(h)$.
\end{proof}

\begin{proof}[Proof of Proposition~\ref{prop8.1}] 
The proposition follows from a special case of the results of Kn\"oller,  \cite{knoeller}. We sketch the idea. 
Let 
$\mathcal M_\Gamma = \Gamma\back \H_n^d$ the Hilbert-Siegel modular variety with respect to $\Gamma$ 
and let $\mathcal M_\Gamma^{BB}$ be its Baily-Borel compactification\footnote{This case was proved earlier by Baily, \cite{baily.HS}. }
Let $\xi_P$ be the point boundary stratum of $\mathcal M_\Gamma^{BB}$ associated to the Siegel parabolic $P$ and let 
$R=\OO_{\xi_P}$ be the corresponding local ring. By the normality of the Baily-Borel compactification, \cite{baily.borel}, $R$ is a normal noetherian local ring with maximal ideal $\mmm=\mmm_{\xi_P}$ and its $\mmm$-adic completion 
$$ \hat R=\varprojlim\limits_{r} R/\mmm^r$$
is also normal and, in particular, an integral domain. 

Let 
$$A = \{\, v\in \mathcal C\mid \tr_{F/\Q}\tr(x v)\ge1, \, \forall x\in \Sb_F\,\}.$$  
This is a `standard kernel', \cite{knoeller}, p.19, \cite{AMRT}, Theorem~5.2 (d). 
Note that $R$ is isomorphic to the ring 
of symmetric Fourier expansions (\ref{F-exp-f}) that are termwise absolutely convergent in sets  of the form
\beq\label{satake-open}
\{\,\tau=u+iv\in \H_n^d \mid v\in t\cdot A\,\},
\eeq
for some $t>0$. Indeed, if $f$ is a holomorphic function on some open neighborhood of $\xi_P$ in $\mathcal M_\Gamma^{BB}$
the pullback to $\H_n^d$ of its restriction to $\mathcal M_\Gamma$ is holomorphic in some open set of the form (\ref{satake-open}), 
cf. Section~6.1 of \cite{AMRT},  and has a symmetric Fourier expansion there. Conversely, the restriction of such a convergent series 
to a sufficiently small neighborhood (\ref{satake-open})  extends to a holomorphic function on an open neighborhood of $\xi_P$ in 
$\mathcal M_\Gamma^{BB}$ since the boundary has codimension $\ge2$. 
Note that $\mmm$ is the ideal of such expansions where $a_f(0)=0$. 

Following Section~2 of \cite{knoeller}, define $\l: \sfs\ra \Z_{\ge1}$ as 
$$\l(x) = \max\{ \,k \mid x = x_1+\dots+x_k, \ x_i\in \sfs\,\}.$$
This function satisfies $\l(x+y)\ge \l(x)+\l(y)$ and $\l(\e\cdot x) = \l(x)$, for all $\e\in \L$. As in Section~3 of \cite{knoeller}, 
let $I_0=R$ and, for $k\ge1$, let 
$$I_k = \{ f\in R\mid a_f(0)=0, \ a_f(x)=0, \forall x\in \sfs \text{\ with $\l(x)<k$}\ \}.$$
Then $I_k$ is an ideal in $R$ and these ideals satisfy 
$I_1 = \mmm$, $I_{k}\subset I_{k-1}$, and $I_k\cdot I_{k'} \subset I_{k+k'}$. 

The following result is the analogue of the Hilfsatz on p.127 of \cite{freitag.keihl} and is proved using standard facts about Poincar\'e series 
with respect to $\L$. 
\begin{lem} 
$$\FFSS_{\L} = \varprojlim\limits_k R/I_k.$$
\end{lem}

Now, as a special case of Satz~3.1.3 of \cite{knoeller}, the filtrations $( I_k)$ and $(\mmm^k)$ define the same topology on $R$ and hence
$$\FFSS_{\L} = \hat R$$
is an integral domain, as claimed. \end{proof}

\section{What Vogan-Zuckerman says about Hodge numbers}\label{section7}

Assume that $1\le d_+<d$, and that $\Gamma$ is neat so that $S= S_\Gamma$  as in (\ref{classical-S}) is a smooth compact complex manifold. 
The Betti cohomology and the Hodge numbers of $S$ can be described 
in term of the $(\g,K)$-cohomology of the space of smooth $K$-finite functions on $\Gamma\back G(\R)^+$, where $G= \SO(V)$.   
Here $G(\R)^+$ is the identity component of $G(\R)$, 
$$G(\R)^+  = \SO_0(m,2)^{d_+}\times \SO(m)^{d-d_+},$$
and 
$$K = (\SO(m)\times \SO(2))^{d_+}\times \SO(m)^{d-d_+},$$
is a maximal compact subgroup. 
Vogan and Zuckerman, \cite{VZ}, describe all irreducible Harish-Chandra modules that can contribute to this cohomology and the degrees in which these 
contribution occurs.  In the case at hand, 
their results imply the vanishing of 
certain Hodge numbers.  

\subsection{ Harish-Chandra modules with non-trivial $(\g,K)$-cohomology for $\SO_0(m,2)$}

In this section,  we work out the results of \cite{VZ} very explicitly (and naively) in this special case.  
This information is also to be found in the literature, cf. \cite{Li}, for example, 
and probably elsewhere, but we feel it might be useful to provide more details.  In particular, we do not know of a reference for the 
picture of Hodge diamond given below. 

We slightly shift notation and let $G= \SO_0(m,2)$ and $K= \SO(m)\times \SO(2)$, with real Lie algebras 
$\g_0$ and $\kk_0$. Let $\g= (\g_0)_\C$ and $\kk = (\kk_0)_\C$ be their complexifications, let $\g_0=\kk_0 +\pp_0$ be the Cartan 
decomposition,  and let $\theta$ be the corresponding Cartan involution, $\theta\vert_{\kk_0} = +1$, $\theta\vert_{\pp_0}=-1$. 

We write elements of $\g_0$ as 
$$X = \bpm X_1&X_2\\{}^tX_2&X_4\epm, \qquad X_1\in \Skew_m(\R), X_4\in \Skew_2(\R), X_2\in M_{m,2}(\R).$$
Here $\kk_0$ is the subalgebra where $X_2=0$ and $\pp_0$ is the subspace where $X_1=0$ and $X_4=0$. 
The element 
$$\hbold = \bpm 1_m&{}\\{}&J\epm, \qquad J=\bpm {}&1\\-1&{}\epm,$$
lies in the center of $K$ and the Harish-Chandra decomposition is 
$$\g = \kk + \pp_++\pp_-,$$
where $\pp_{\pm}$ are the $\pm i$-eigenspaces of $\text{\rm Ad}(\hbold)\vert_\pp$. 
Concretely, these subspaces are given by 
$$\pp_{\pm} = \{\ X\in \pp\mid X_2 =  (x_2, \pm i x_2), \ x_2\in \C^m\ \},$$
where we note that $(x_2, \pm i x_2) i J = \pm  (x_2, \pm i x_2)$. 
A Cartan subalgebra $\mathfrak t_0$ of $\g_0$ is given by 
$$\tbold(\abold)=\tbold(a_0, a_1, \dots, a_{m'})  = \bpm \diag(a_1 J, \dots, a_{m'} J, 0_*)&{}\\{}& a_0 J\epm, \qquad J=\bpm {}&1\\-1&{}\epm,$$
where $a_j\in \R$, 
$m' = \left[\frac{m}{2}\right]$, 
and $0_*$ indicates an extra $0$ when $m$ is odd.

The representations $A_\qqq$ are irreducible Harish-Chandra modules associated to $\theta$-stable parabolic subalgebras $\qqq$ of $\g$.  
These subalgebras are constructed as follows, \cite{VZ}, Section 2.  For integers $r$, $s$ with $0\le r\le m'$, $0\le s\le 1$, $r+s<\frac12 m$, let 
$$x_{r,s}(\abold) =i\,\tbold(a_0, a_1, \dots, a_{m'})\ \in i \,\kk_0,$$ 
where $a_j \ne 0$ for $1\le j \le r$, $a_j=0$ for $j>r$,  $a_0=0$ for $s=0$, and $a_0\ne 0$ for $s=1$. 
The endomorphism $\text{\rm ad}(x_{r,s})$ of $\g$ is diagonalizable and the subalgebra
$$\qqq =  \text{sum of the $\text{\rm ad}(x_{r,s})$-eigenspaces with eigenvalue $\mu\ge0$.}$$
is a $\theta$-stable parabolic subalgebra with decomposition 
$\qqq=\mathfrak l + \mathfrak u$ where $\mathfrak l$ is the sum of the eigenspaces with $\mu=0$  and $\mathfrak u$ is the sum of the 
eigenspaces with $\mu>0$. 

The representation $A_{\qqq}$ associated to $\qqq$ is characterized by Theorem~2.5 of \cite{VZ}.  
As explained in Section 6 of \cite{VZ}, there is a Hodge type decomposition
\beq\label{gK-Hodge}
H^i(\g,K;A_{\qqq}) =\bigoplus_{p+q=i} H^{p,q}(\g,K;A_{\qqq})
\eeq
of the $(\g,K)$-cohomology. 
The main fact that we need is the following result, extracted from Proposition~6.19 of \cite{VZ}. 
\begin{prop}  For a $\theta$-stable parabolic $\qqq = \mathfrak l + \mathfrak u$,  let $R_{\pm} = \dim(\mathfrak u \cap \pp_{\pm})$. 
Then 
$$H^{p,q}(\g,K;A_{\qqq}) =0$$ 
unless $p-q = R_+-R_-$. 
\end{prop} 

The next result records the possible values for $(R_+,R_-)$. 

\begin{prop} Suppose that $\qqq$ is is constructed, as above, from $x_{r,s}(\abold)$. \hfb 
(i) If $s=0$, then $(R_+,R_-) = (r,r)$. \hfb
(ii) If $s=1$,  let $\delta_+$ (resp. $\delta_-$) be the number of $j$'s for which $a_j + a_0=0$ (resp. $a_j  -a_0=0$). 
Then 
$$(R_+,R_-) = \begin{cases} (r - \delta_+, m-r-\delta_-) &\text{if $a_0>0$}\\
(m-r-\delta_+,r-\delta_-)&\text{if $a_0<0$.}
\end{cases}
$$
In particular, $0\le \delta_++\delta_- \le r \le \left[\frac{m}2\right]$ and 
$$R_++R_- = m-\delta_+-\delta_-  \ge m-\left[\frac{m}2\right].$$
\end{prop}

\begin{proof} We want to calculate $\dim(\mathfrak u \cap \pp_{\pm})$ and so we consider the action of $\text{ad}(x_{r,s}(\abold))$ and $\text{Ad}(J)$ on 
an element of $\pp$. Write 
$$\text{\rm ad}(x_{r,s}(\abold)):\ X_2 = \bpm A_1\\ \vdots \\ A_r\\ \a_1\\ \vdots\\ \a_{m-2r}\epm \ \longmapsto \bpm  a_1\, iJ A_1 - a_0\,A_1 \,iJ\\
\vdots \\ a_r\, iJ A_r - a_0\,A_r \,iJ\\ - a_0\,\a_1\, iJ\\ \vdots\\ - a_0 \,\a_{m-2r}\,iJ\epm.
$$
for $2\times 2$ blocks $A_j$ and row vectors $\a_k$.  The eigenvalues of the transformation 
$$A \mapsto a\, iJ A - a_0\,A \,iJ, \qquad A\in M_2(\C), \ a, a_0\in \C$$
are $\pm(a-a_0)$ and $\pm (a+ a_0)$. Corresponding eigenvectors are
$$\bpm -1&i\\i&1\epm, \quad \bpm 1&-i\\i&1\epm, \quad \bpm 1&i\\ -i&1\epm, \quad \bpm -1&-i\\-i&1\epm,$$
where we note that the $\pm (a-a_0)$-eigenvectors satisfy $A iJ = - A$ and the $\pm (a+a_0)$-eigenvectors 
satisfy $A iJ = A$. 
For a row vector $\a$, the eigenvalues of the transformation $\a \mapsto -a_0\,\a\,iJ$ are $\pm a_0$ with corresponding eigenvectors $\pm(1,-i)$. 
Thus the $a_0$-eigenvector satisfies $\a \,iJ = -\a$ and the $-a_0$-eigenvector satisfies $\a\,iJ = \a$. 
\hfb 
Suppose that $s=1$ so that $a_0\ne 0$ and that $x_{r,1}(\abold)$ is regular, i.e., that $a_j\pm a_0\ne 0$ for $1\le j\le r$.  Then in $j$th block, precisely two eigenspaces have positive eigenvalue
and one of them lies in $\pp_+$ and the other in $\pp_-$. Thus each such block contributes $1$ to both $R_+$ and $R_-$.  
The eigenvalues of the transformation $\a \mapsto -a_0\,\a\,iJ$ are $\pm a_0$.  For $a_0>0$ (resp. $a_0<0$), such a row contributes $1$ to $R_-$ 
(resp. $R_+$). This proves (ii) in the regular case.  \hfb 
Now suppose that $s=1$ and let $\delta_+$ (resp. $\delta_-$) be the number of $j$'s for which $a_j + a_0=0$ (resp. $a_j  -a_0=0$). 
Now the $A_j$ blocks with $a_j+a_0=0$ do not make a contribution to $R_+$, since the corresponding eigenvalue $\mu=0$, while 
the $A_j$ blocks with $a_j-a_0=0$ do not make a contribution to $R_-$. This proves (ii) in general. 
We omit the easier case (i). 
\end{proof}

{\bf Remark.} 
Note that when $s=1$ and $x_{r,1}(\abold)$ is regular,  the whole space $\pp$ is spanned by non-zero eigenspaces. Thus we have $\mathfrak l \subset \kk$ so that
the $A_{\qqq}$'s are discrete series representations, \cite{VZ}, p. 58, and contribute only to the cohomology in the middle dimension, with Hodge numbers 
$(r,m-r)$, for $a_0>0$ and $(m-r,r)$, for $a_0<0$. 

We thus have the following vanishing result for Hodge numbers, which plays an essential role in our argument. 
\begin{cor}  If $H^{p,q}(\g,K;A_{\qqq})\ne 0$ for some $(p,q)$ with $p\ne q$, then 
$p+q \ge m - \left[\frac{m}2\right].$
\end{cor} 

\subsection{Global consequences for Hodge numbers}

Now return to the global situation where $S = \Gamma\back D^+ = \Gamma\back G(\R)^+/K$,  
for $\Gamma$ neat and $1\le d_+<d$, or $d_+=d$ and $V$ is anisotropic\footnote{This require that $V_v$ be anisotropic at some finite place of $F$ and hence can only occur when $\dim_F V\le 4$.}.
Following the discussion in the introduction of \cite{VZ}, write
$$L^2(\Gamma\back G)\simeq \bigoplus_{\pi \in \hat G} m_\pi \,\mathcal H_\pi,$$
a Hilbert space direct sum of irreducible unitary representations of $G$ with finite multiplicities. 
The cohomology of $S$ is then described as
$$H^*(S,\C) \simeq \bigoplus_{\pi\in \hat G} m_\pi \, H^*(\g,K; \mathcal H^K_\pi),$$
where $\mathcal H^K_\pi$ is the Harish-Chandra module of $\pi$. Vogan and Zuckerman, \cite{VZ}, Theorem~4.1, show that  $H^*(\g,K; \mathcal H^K_\pi)\ne 0$ 
only if $\mathcal H^K_\pi \simeq A_{\qqq}$ for some $\qqq$. 
Thus, via the Kunneth formula for relative Lie algebra cohomology and (\ref{gK-Hodge}), we have the following. 

\begin{cor}\label{cor-better-van}  Suppose that $H^{(p,q)}(S) \ne 0$ for some $(p,q)$ with $p\ne q$. Then $p+q \ge m- \left[\frac{m}2\right].$
In particular $H^{2k-1}(S,\C) = 0$ if $2k-1<m- \left[\frac{m}2\right]$ and the intermediate Jacobian $J^k(S)=0$ in this range. 
\end{cor} 
Note that one should be able to slightly extend this vanishing range using relations between archimedean components of automorphic representations
coming from Arthur's classification, but we will not 
need such an improvement. 

\section{Weighted cycles and `natural' cycles}\label{section10}

In this section, we extend the discussion of Sections 2--5 of \cite{K.duke} to the case of $d_+>1$.  In \cite{K.duke}, Definition~5.2, the weighted cycles were defined 
in terms of the `natural' cycles given by sub-Shimura varieties.  An 
expression for them in terms of connected or classical cycles was then proved in Proposition~5.5 of loc. cit.
Here in our definition (\ref{def-wt-cycle}) of weighted cycles in Section~\ref{section5}, we have taken the 
expression in terms of connected/classical cycles as the starting point, since this expression is what is needed for the intersection theory 
calculations of Section~\ref{section4}.  
In the present section, we prove that the weighted cycles defined by the analogue of Definition~5.2 of \cite{K.duke} 
coincide with the ones defined in Section~\ref{section5} above. 
Especially in the case $n=m$, this requires some care 
about connected components and so we give a detailed discussion.

Having fixed a connected component $D^+$ in Section~\ref{section2} above, we can index the components $D^\e$ of $D$ 
by collections $\e = (\e_1, \dots, \e_{d_+})$, where $\e_j= \pm1$.   

For a totally positive subspace $W\subset V$ with $\dim_F W= n\le m$, we let $H = G_W$ be the pointwise stabilizer of $W$ in $G$, so that $H\simeq R_{F/\Q} \GSpin(W^\perp)$. 
The space $W^\perp_j = W^\perp\tt_{F,\s_j} \R$, for $1\le j \le d_+$,  has signature $(m-n,2)$. 
We let $D_W^{(j)}$ be the space of oriented negative $2$-planes in $W_j^\perp$, let 
$$D_W = \prod_j D_W^{(j)},$$
and let $D_W^\e = D_W\cap D^\e$.  Thus, $D_W^+ = D_W\cap D^+$, as in (\ref{DpW}),  and $D_W$ has $2^{d_+}$ components.   

Note that, if $n<m$, the group $H(\R)$ also has $2^{d_+}$ components and acts transitively on $D_W$, whereas, if $n=m$, the group
$$H(\R) \simeq \GSpin(2)^d$$
is connected and acts trivially on the finite set of points $D_W$.  This distinction will lead to slight differences in the treatment of the two cases.

For $g\in G(\A_f)$, the codimension $d_+n$ cycle
$$Z(W,g,K)^\nat:=H(\Q) \back D_W \times H(\A_f)/K_{H,g}\  \lra \   G(\Q)\back D\times G(\A_f)/K, \qquad [z,h] \ \mapsto [z, hg]$$
is called a `natural' cycle in \cite{K.duke}. Here $K_{H,g}= H(\A_f)\cap g K g^{-1}$. 
On the other hand, for any $g_0\in G(\A_f)$, if we let $\Gamma'_{g_0} = G(\Q)_+\cap g_0 K g_0^{-1}$ and let $\Gamma_{g_0}$ be its image in $\SO(V)$, 
then the connected cycle
defined in (\ref{basic-cycle}) is 
$$Z(W)_{\Gamma_{g_0}} = \pi_{\Gamma_{g_0}}(D_W^+)\ \subset \Gamma_{g_0}\back D^+.$$

We next describe the relation between these cycles. 

For each $\e$, choose $\gamma_\e\in G(\Q)$ such that 
$$\gamma_\e\, D^\e = D^+.$$
Writing 
$$G(\A_f)= \bigsqcup_j G(\Q)_+ g_j K,$$
we have an isomorphism
\beq\label{compo-iso}
G(\Q)\back D\times G(\A_f)/K \isoarrow \bigsqcup_j \Gamma_{g_j}\back D^+,
\eeq
where, if $z\in D^\e$ and $\gamma_\e g = \gamma^{-1} g_j k$ with $\gamma\in G(\Q)_+$, then 
$$[z,g] \ \mapsto \pi_{\Gamma_{g_j}}(\gamma\gamma_\e z).$$

As in Lemma~4.1 of \cite{K.duke}, write
\beq\label{HA-double}
H(\A_f) = \bigsqcup_i H(\Q)_+ h_i K_{H,g}.
\eeq
If $n<m$, then we can choose $\gamma_{\e,H}\in H(\Q)$ with $\gamma_{\e,H} D^\e_W = D^+_W$ and obtain
\beq\label{H-compo-iso}
H(\Q) \back D_W \times H(\A_f)/K_{H,g} \isoarrow  \bigsqcup_i \Gamma_{H, i}\back D^+_W,
\eeq
where $\Gamma_{H,i} = H(\Q)_+\cap h_i K_{H,g} h_i^{-1} = H(\Q)_+ \cap h_i g K g^{-1} h_i^{-1}$. 
For each $i$, write $h_ig = \gamma_i^{-1}g_j k_i$ where $\gamma_i\in G(\Q)_+$ and $j=j(i)$ depend on $h_i$. 
Then the point $\pi_{\Gamma_{i,H}}(z)$, $z\in D_W^+$ on the right side of (\ref{H-compo-iso}), with preimage $[z,h_i]$ on the left side, 
maps to the point $\pi_j(\gamma_i z)$ on the right side of (\ref{compo-iso}). Note that $\gamma_i z\in D_{\gamma_i W}^+$.  This proves the following. 
\begin{lem} If $n<m$, then, under the isomorphism (\ref{compo-iso}),
\beq\label{new.lem.4.1-a}
Z(W,g,K)^\nat = \sum_i Z(\gamma_i W)_{\Gamma_j}, 
\eeq
where, for  coset representatives $h_i$ as in (\ref{HA-double}), 
$h_ig = \gamma_i^{-1}g_j k_i$, with $\gamma_i\in G(\Q)_+$, $k_i\in K$,  and $j=j(i)$ depending on $h_i$.
\end{lem} 

Now suppose that $n=m$ so that $H(\Q)=H(\Q)_+$ and 
\beq\label{H-compo-iso-bis}
H(\Q) \back D_W \times H(\A_f)/K_{H,g} =  D_W\times H(\Q)_+\back H(\A_f)/K_{H,g} \isoarrow \bigsqcup_i D_W.
\eeq
For each $i$ and $\e$, write
\beq\label{HA-compos-bis}
\gamma_\e h_i g = \gamma^{-1}_{i,\e} g_j k,
\eeq
where $j$,  $\gamma_{i,\e}\in G(\Q)_+$, and $k$ depend on $h_i$ and $\gamma_\e$. 
The point $z\in D_W^\e$ in the $i$-th part of the right side of (\ref{H-compo-iso-bis}), with preimage $[z,h_i]$ on the left side, maps to the point 
$\pi_{\Gamma_j}(\gamma_{i,\e}\gamma_\e z)$ on the right side of (\ref{compo-iso}), via
$$[z,h_i] \mapsto [z, h_ig] = [\gamma_\e z, \gamma_\e h_i g] = [\gamma_{i,\e}\gamma_\e z, g_j] \mapsto \pi_{\Gamma_{g_j}}(\gamma_{i,\e}\gamma_\e z).$$
  
\begin{lem}\label{new.lem.4.1-bis}  If $n=m$, then,  under the isomorphism (\ref{compo-iso}),  
\beq\label{new.lem.4.1}
Z(W,g,K)^\nat = \sum_{\e}\sum_i   Z(\gamma_{i,\e}\gamma_\e W)_{\Gamma_{g_j}}.
\eeq
\end{lem}
These expressions for the natural cycles in terms of connected cycles are the analogue of parts (i) and (ii) of Lemma~4.1 of \cite{K.duke}. 

Next we consider the weighted cycles.   
The discussion of Section 5 of \cite{K.duke} carries over with minor changes. 
First suppose that $T\in \Sym_n(F)$ is totally positive definite and that $\O_T(F)\neq \emptyset$, where $\O_T$ is the hyperboloid 
in $V^n$, as in (\ref{def-O-T}).  Fix $x_0\in \O_T(F)$. Then  $\O_T(\A_f) = G(\A_f)\cdot x_0$ and, for a $K$-invariant Schwartz function 
$\ph\in S(V(\A_f)^n)$, write
$$(\text{supp}\, \ph)\cap \O_T(\A_f) = \bigsqcup_r \ K\cdot \xi_r^{-1}\cdot x_0,$$
with $\xi_r\in G(\A_f)$ as in (5.4) of \cite{K.duke}. 
Consider the weighted sum of natural cycles 
\beq\label{def-nat-wt-cycles}
Z(T,\ph,K)^{\nat} := \sum_r \ph(\xi_r^{-1} x_0)\,Z(W(x_0),\xi_r,K)^\nat,
\eeq
as in Definition~5.2 of \cite{K.duke}.  

\begin{prop}\label{prop10.2} The two definitions of weighted special cycles  (\ref{def-wt-cycle}) and (\ref{def-nat-wt-cycles}) coincide,
$$Z(T,\ph,K) =Z(T,\ph,K)^{\nat}.$$
\end{prop}
\begin{proof}  We only give the proof in the case $n=m$ and $T\in \Sym_m(F)_{>0}$ totally positive definite. The remaining cases follow 
by similar arguments. 
Let $\nu: G \ra R_{F/\Q}\G_m$ denote the spinor norm and note that 
$$\nu(G(\Q)) = \{ \a\in F^\times \mid \s_j(\a)>0, \forall j>d_+\ \}.$$
This implies that $G(\Q)_+$ has $2^{d_+}$-orbits in $\O_T(F)$. These can be indexed as follows. Fix an (ordered) $F$-basis for $V$. This basis 
determines an orientation for $V_j$ for all $j$. An $n$-frame $x\in \O_T(F)$ determines an orientation of $W(x)\tt_{F,\s_j}\R$ for all $j$ and 
hence, since an orientation of $V_j$ has been fixed,  an orientation for $W(x)^\perp \tt_{F,\s_j}\R$. Thus we obtain a point in $D^\e=D^{\e(x)}$. 
This defines a function $\e:\O_T(F) \ra (\pm1)^{d_+}$ which distinguishes the $G(\Q)_+$-orbits.

We can write 
\beq\label{old-wt-cycle}
Z(T,\ph,K) = \sum_\e\sum_j \sum_{\substack{x\in \O_T(F)^\e \\ \snass \mod \Gamma_{g_j}}} \ph(g_j^{-1} x) \, Z(W(x))_{\Gamma_{g_j}},
\eeq
where we note that the only $x$'s that contribute to the $j$-th summand are those in the set 
$$\O_T(F)^\e  \cap \big(\ g_j\text{supp}(\ph)\cap \O_T(\A_f)\ \big) = \bigsqcup_r  \ G(\Q)_+\gamma_\e\, x_0 \cap g_j  K \xi_r^{-1}\, x_0.$$
Here we suppose that $\e(x_0) = (+1, \dots, +1)$. 
Thus we have 
$$Z(T,\ph,K) = \sum_\e\sum_j \sum_r\sum_{\substack{x\in \O_T(F)^\e  \cap g_j \cdot K\cdot \xi_r^{-1}\cdot x_0 \\ \snass \mod \Gamma_{g_j}}} \ph(g_j^{-1} x) \, Z(W(x))_{\Gamma_{g_j}}.$$

Now, by Weil's Lemma, \cite{K.duke}, Lemma~5.7 (ii), there is a bijection
\beq\label{Weil-lem}
\begin{matrix} \text{$\Gamma_{g_j}$-orbits in $G(\Q)_+\gamma_\e\, x_0 \cap g_j  K \xi_r^{-1}\, x_0$,}\\
\nass
\nass
\updownarrow\\
\nass
\nass
H(\Q)\back \big(\ H(\A_f) \cap \gamma_\e^{-1} G(\Q)_+ g_j K \xi_r^{-1}\ \big)/ K_{H, \xi_r}, 
\end{matrix}
\eeq
given by 
$$\Gamma_{g_j}\cdot x \ \longmapsto \ H(\Q) \gamma_\e^{-1}\gamma^{-1} g_j k \xi_r^{-1} K_{H,\xi_r}, \qquad x  =\gamma \gamma_\e x_0 = g_jk\xi_r^{-1} x_0.$$
Here note that the two expressions for $x$ insure that $\gamma_\e^{-1}\gamma^{-1} g_j k \xi_r^{-1}\in H(\A_f)$ and that 
$$\ph(g_j^{-1} x) = \ph(\xi_r^{-1} x_0).$$
We can then rearrange the sum as 
$$Z(T,\ph,K) = \sum_r\ph(\xi_r^{-1} x_0)\sum_\e\sum_j \sum_{\substack{x\in \O_T(F)^\e  \cap g_j \cdot K\cdot \xi_r^{-1}\cdot x_0 \\ \snass \mod \Gamma_{g_j}}} Z(W(x))_{\Gamma_{g_j}}.$$
Now, for each $\e$ and $r$, 
$$H(\A_f) = \bigsqcup_j \ H(\A_f) \cap \gamma_\e^{-1} G(\Q)_+ g_j K \xi_r^{-1},$$
where, we see that subsets on the right side give the elements of $h\in H(\A_f)$ such that $\gamma_\e h \xi_r$ lies in $G(\Q)_+ g_j K$.   
Returning to (\ref{HA-double}) with $g=\xi_r$, and writing a double coset representative as $h_i = \gamma_\e^{-1}\gamma_{i,\e}^{-1} g_j k\xi_r^{-1}$, 
the corresponding $x$ is $\gamma_{i,\e}\gamma_\e x_0$. 
Indexing the sum by the double cosets, we have 
\begin{align*}
Z(T,\ph,K) &= \sum_r\ph(\xi_r^{-1} x_0)\sum_\e\sum_i 
Z(W(\gamma_{i,\e}\gamma_\e x_0))_{\Gamma_{g_j}}\\
\nass
{}&= \sum_r\ph(\xi_r^{-1} x_0) \, Z(W(x_0),\xi_r,K)^\nat\\
\nass
{}&=Z(T,\ph,K)^{\nat},
\end{align*}
by Lemma~\ref{new.lem.4.1-bis}. 
\end{proof}

\end{document}